\documentclass[aop,MSNbibl,dvips]{arximspdf}
\input xypic

\doi{10.1214/13-AOP858} %kopijuoti is PTS
\volume{42}
\issue{3}
\pubyear{2014}
\firstpage{1212}
\lastpage{1256}

\makeatletter
\newcommand{\eqref}[1]{(\ref{#1})}
\newtheorem{theorem}{Theorem}[section]
\newtheorem{proposition}[theorem]{Proposition}
\newtheorem{lemma}[theorem]{Lemma}
\newtheorem{corollary}[theorem]{Corollary}

\newtheorem{conjecture}[theorem]{Conjecture}

\newproclaim{definition}[theorem]{Definition}
\newproclaim{example}[theorem]{Example}
\newproclaim{remark}{Remark}
\newcommand{\Z}{\mathbb{Z}}
\newcommand{\R}{\mathbb{R}}
\newcommand{\N}{\mathbb{N}}
\newcommand{\B}{\mathcal{B}}

\newcommand{\vi}{\mathbf{i}}
\newcommand{\vx}{\mathbf{x}}
\newcommand{\vy}{\mathbf{y}}
\newcommand{\vt}{\mathbf{t}}

\renewcommand{\P}{\mathbb{P}}
\newcommand{\envP}{Q}
\newcommand{\Pl}{\mathbf{P}}
\newcommand{\E}{\mathrm{E}}
\newcommand{\El}{\mathbf{\E}}
\newcommand{\Var}{\operatorname{Var}}
\newcommand{\F}{\mathcal{F}}
\renewcommand{\S}{\mathcal{S}}
\newcommand{\Ai}{\operatorname{Airy}}
\newcommand{\Aifunc}{\operatorname{Ai}}
\newcommand{\sym}{\operatorname{Sym}}
\newcommand{\symz}{\mathfrak{Z}}

\makeatother

\begin{document}
\begin{frontmatter}

\title{The intermediate disorder regime for directed polymers in
dimension $1+1$\thanksref{T1}}
\runtitle{Intermediate disorder regime for directed polymers}
\thankstext{T1}{Supported by the
Natural Sciences and Engineering Research Council of Canada.}

\begin{aug}
\author[A]{\fnms{Tom} \snm{Alberts}\corref{}\ead[label=e1]{alberts@caltech.edu}},
\author[B]{\fnms{Konstantin} \snm{Khanin}}
\and
\author[B]{\fnms{Jeremy} \snm{Quastel}}
\runauthor{T. Alberts, K. Khanin and J. Quastel}
\affiliation{California Institute of Technology, University of
Toronto\break
and University of Toronto}
\address[A]{T. Alberts\\
Department of Mathematics\\
California Institute of Technology\\
MC 253-37\\
1200 E California Blvd.\\
Pasadena, California 91125\\
USA} %adresu isvedimo komanda gale!
\address[B]{K. Khanin\\
J. Quastel\\
Department of Mathematics\\
University of Toronto\\
Room 6290, 40 St. George St.\\
Toronto, Ontario M5S 2E4\\
Canada}
%University of Toronto}
\end{aug}

% HISTORY:
\received{\smonth{4} \syear{2012}}
\revised{\smonth{2} \syear{2013}}

% ABSTRACT
%
\begin{abstract}
We introduce a new disorder regime for directed polymers in dimension
$1+1$ that sits between the weak and strong disorder regimes. We call
it the \textit{intermediate disorder regime}. It is accessed by scaling
the inverse temperature parameter $\beta$ to zero as the polymer length
$n$ tends to infinity. The natural choice of scaling is $\beta_n:=
\beta n^{-1/4}$. We show that the polymer measure under this scaling
has previously unseen behavior. While the fluctuation exponents of the
polymer endpoint and the log partition function are identical to those
for simple random walk ($\zeta= 1/2, \chi= 0$), the fluctuations
themselves are different. These fluctuations are still influenced by
the random environment, and there is no self-averaging of the polymer
measure. In particular, the random distribution of the polymer endpoint
converges in law (under a diffusive scaling of space) to a random
absolutely continuous measure on the real line. The randomness of the
measure is inherited from a stationary process $A_{\beta}$ that has the
recently discovered \textit{crossover distributions} as its one-point
marginals, which for large $\beta$ become the GUE Tracy--Widom
distribution. We also prove existence of a limiting law for the
four-parameter field of polymer transition probabilities that can be
described by the stochastic heat equation.

In particular, in this weak noise limit, we obtain the convergence of
the point-to-point free energy fluctuations to the GUE Tracy--Widom
distribution. We emphasize that the scaling behaviour obtained is
universal and does not depend on the law of the disorder.

\end{abstract}

% KEYWORDS
\begin{keyword}[class=AMS]
\kwd{60F05}
\kwd{82C05}
\end{keyword}
\begin{keyword}
\kwd{Directed polymers}
\kwd{near-critical scaling limits}
\kwd{$U$-statistics}
\kwd{KPZ equation}
\end{keyword}

\end{frontmatter}

%s1 #&#
\section{Introduction}\label{sec1}

The problem of directed polymers in a random environment was first
studied in \cite{HuseHenley} and received its first mathematical
treatment in \cite{ImSpen:diffusion}. Since then it has received
considerable attention in the statistical physics and probability
communities; see \cite{ComSY:review,HalpinHealy_Zhang} for reviews. In
the setting of the $d$-dimensional integer lattice, the polymer measure
is a random probability measure on paths of $d$-dimensional nearest
neighbour lattice walks. The randomness of the polymer measure is
inherited from an i.i.d. collection of random variables placed on the
sites of $\Z_+ \times\Z^d$. Collectively these variables are called
the random environment. Given a fixed environment $\omega\dvtx  \Z_+
\times
\Z^d \to\R$, the energy of an $n$-step nearest neighbour walk $S$ is
\[
H_n^{\omega}(S) = \sum_{i=1}^n
\omega(i, S_i).
\]
The polymer measure on such walks is then defined in the usual Gibbsian
way by
\[
\Pl_{n, \beta}^{\omega}(S) = \frac{1}{Z_{n}^{\omega}(\beta)} e^{\beta
H_n^{\omega}(S)}
\Pl(S),
\]
where $\beta> 0$ is the inverse temperature, $\Pl$ is the symmetric
simple random walk measure on paths started at the origin and
$Z_{n}^{\omega}(\beta)$ is the partition function
\[
Z_{n}^{\omega}(\beta) = \Pl \bigl[ e^{\beta H_n^{\omega}(S)} \bigr].
\]
The random environment is a probability measure $Q$ on the space of
environments $\Omega= \{ \omega\dvtx  \Z_+ \times\Z^d \to\R\}$. We let
$Q$ be a product measure so that the variables $\omega(i,z)$ are
independent and identically distributed, and we make the assumption
that the $\omega$ have moments of all orders and that
\[
\lambda(\beta):= \log\envP e^{\beta\omega} < \infty,
\]
at least for $\beta$ sufficiently small. In the conclusion, we will
describe work in progress for relaxing this assumption. Note that
throughout we use $\Pl$ and $Q$ to denote expectation as well as probability.

The overall goal of the subject is to study the behavior of the polymer
as $\beta$ and~$d$ vary and $n$ gets large. At $\beta= 0$ the polymer
measure is the simple random walk; hence the walk is entropy dominated
and exhibits diffusive behavior. For $\beta$ large the polymer measure
concentrates on paths with high energy and the diffusive behavior is no
longer guaranteed. Entropy domination of the measure is called \textit
{weak disorder}, and energy domination is called \textit{strong
disorder}. The precise separation between these two regimes is defined
in terms of the positivity of the limit of the martingale $e^{-n
\lambda
(\beta)} Z_n^{\omega}(\beta)$, as $n \to\infty$. The weak disorder
regime consists of $\beta$ for which
\[
\lim_{n \to\infty} e^{-n \lambda(\beta)} Z_n^{\omega}(
\beta) > 0,
\]
whereas if the limit is zero, then $\beta$ is said to be in the strong
disorder regime. For $d \geq3$, it was shown early on \cite
{ImSpen:diffusion,bolthausen:note} that weak disorder holds for small
$\beta$. Later,\vadjust{\goodbreak} Comets and Yoshida \cite{ComY:diffusive_weak} showed
that in every dimension there is a critical value $\beta_c$ such that
weak disorder holds for $0 \leq\beta< \beta_c$ and strong disorder
for $\beta> \beta_c$. In addition, for $d=1$ and $2$ they prove that
$\beta_c = 0$. In the rest of this paper we focus exclusively on the
case $d=1$ so that all positive $\beta$ are in the strong disorder regime.

Understanding the polymer behavior in the strong disorder regime is one
of our main goals. Strong disorder manifests itself in a variety of
ways that have become more evident in recent years. Arguably the most
well-known phenomenon is superdiffusivity of the paths under the
polymer measure. This is usually expressed through an exponent $\zeta$,
and although there is no commonly agreed upon definition of $\zeta$ in
the literature, it is roughly meant to be the exponent such that
\[
|S_n| \sim n^{\zeta}
\]
as $n \to\infty$, for ``typical'' realizations of $\omega$. For $d=1$
the long-standing conjecture is $\zeta= 2/3$, but at present has been
obtained only in models with specific weights and sometimes boundary
conditions. The best result is by Sepp\"{a}l\"{a}inen \cite
{Timo:exponent} for a model with specific weights and boundary
conditions. There are also upper and lower bounds (neither one sharp)
given for certain special models \cite{ComY:brownian_polymers,mejane:volume_bound,petermann:thesis,wuthrich:fluctuation_results,wuthrich:superdiffusive_behavior}. The picture is very different from
that of simple random walk where $\zeta= 1/2$, and the polymer
endpoint is roughly uniformly distributed on intervals of length $\sqrt {n}$. For positive $\beta$ the polymer is \textit{localized} and most
of the endpoint density sits in a relatively small region around a
random point at distance $n^{2/3}$ from the origin. The size of this
region is of much smaller order than $n^{2/3}$. In fact, it is believed
that the variance of the polymer endpoint is order one. Carmona and Hu
\cite{CH:Gaussian_partition} and Comets et al. \cite
{ComSY:localization_strong} showed that there is a constant $c_0 =
c_0(\beta) > 0$ such that the event
\[
\limsup_{n \to\infty} \max_{x \in\Z}
\Pl_{n, \beta}^{\omega
}(S_n = x) \geq c_0
\]
has $Q$ probability one. This phenomenon is called \textit{strong
localization}. It is in stark contrast to the simple random walk case
where the supremum decays like $n^{-1/2}$.

Strong disorder also has an effect on the large time behavior of the
partition function. First, for $\beta> 0$, there is the well-known inequality
%
%e1 #&#
\begin{eqnarray}
\label{annealing_bound} \rho(\beta)&:=& \lim_{n \to\infty} \frac{\log Z_{n}^{\omega
}(\beta)}{n} =
\lim_{n \to\infty} \frac{Q \log Z_{n}^{\omega}(\beta)}{n} < \lim_{n
\to\infty}
\frac{\log Q Z_n^{\omega}(\beta)}{n}
\nonumber
\\[-8pt]
\\[-8pt]
\nonumber
&=& \lambda(\beta)
\end{eqnarray}
between the quenched and annealed free energies (the second equality is
by a subadditivity argument and some concentration estimates; see, for
example, \cite{CH:Gaussian_partition,ComSY:localization_strong}). The
inequality is partially the standard annealing bound, but the fact that
it is strict is a feature of strong disorder that was proved in $d=1$
by Comets et al. \cite{ComV:cascades_polymers}.\vadjust{\goodbreak} Quantitative bounds on
the size of the gap were later proved in \cite
{lacoin:free_energy_bounds}. When strict inequality holds, $\beta$ is
said to be in the \textit{very strong disorder} regime, and hence in
$d=1$ very strong disorder and strong disorder are equivalent. From
\eqref{annealing_bound} the leading term behavior of the log of the
partition function is $\rho(\beta) n$, and the randomness is
conjectured to appear through a lower order term
%
%e2 #&#
\begin{equation}
\log Z_n^{\omega}(\beta) = \rho(\beta)n + c(\beta)
n^{\chi} X.
\end{equation}
The fluctuation exponent $\chi$ is believed to be $1/3$ for $d=1$, and
the random fluctuations $X$ are expected to converge to the
Tracy--Widom GOE distribution \cite{TracyWidom:level_spacing} arising
from the analogous asymptotics for the largest eigenvalue from the
Gaussian orthogonal ensemble. Observe that the conjectured values of
$\zeta$ and $\chi$ satisfy the simple relation (KPZ relation)
%
%e3 #&#
\begin{equation}
\label{convexity_relation} \chi= 2 \zeta- 1.
\end{equation}
Versions of this were recently proved rigorously in \cite
{chatterjee:exponents} and \cite{ad:exponents}, under the assumption
that the exponents exist as appropriate limits. In the $\beta= \infty$
case of last passage percolation Johannson has established, for a
particular distribution for the environment variables, that the scaling
exponents $\chi= 1/3$ and $\zeta= 2/3$ are correct and that the
fluctuations are of Tracy--Widom type, but rigorous mathematical proofs
remain elusive in the $\beta< \infty$ case of directed polymers.
Sepp\"
{a}l\"{a}inen \cite{Timo:exponent} managed to give proofs of the
exponents for the log-gamma model with a specific choice of the
environment and some boundary conditions, but the more general case
remains open.

As the description above indicates, much of the conceptual picture for
polymers in the strong disorder regime is understood, but little of it
is rigorously proved. In this paper we introduce a new disorder regime
that is interesting in its own right, and for which we are also able to
prove many results. We call this regime the \textit{intermediate
disorder regime}, and it is exclusive to the $d=1$ case. It corresponds
to a scaling of the inverse temperature with the length of the polymer.
The name is chosen because it sits between weak and strong disorder and
features of both are present. While the fluctuation exponents for
intermediate disorder coincide with those for weak disorder ($\zeta=
1/2, \chi= 0$), the fluctuations themselves, as in the strong disorder
regime, are \textit{not} decoupled from the random environment. In
particular, in contrast to the weak disorder case, the polymer measure
does not converge to a single deterministic limit. Instead, under the
diffusive scalings, the \textit{law} of the random polymer measure
converges to a limiting universal law, universal in the sense that it
does not depend on the particular distribution of the environment. In
the $1+1$-dimensional case we will prove: \newline

\textit{Under the scaling $\beta_n = \beta n^{-{1}/{4}}$ the following is true}:
\begin{itemize}
\item The partition function fluctuation exponent $\chi$ is $0$, and
the partition function
\[
e^{-n \lambda(\beta n^{-1/4})} Z_n^{\omega} \bigl(\beta n^{-{1}/{4}}
\bigr)\vadjust{\goodbreak}
\]
converges in law to a nondegenerate random variable $\mathcal
{Z}_{\sqrt
{2} \beta}$. The limiting random variable is square integrable, and its
Wiener chaos decomposition is explicit; see \eqref{Zn_wiener_chaos}.
\item The path fluctuation exponent $\zeta$ is $1/2$, and the law of
the distribution of $S_n/\sqrt{n}$ converges to a random density on the
real line. In particular this implies the absence of localization. More
precisely, there is a random local limit theorem for the endpoint density
\[
\biggl\{ x \mapsto\frac{\sqrt{n}}{2} \Pl_{n, \beta_n}^{\omega}(S_n
= x \sqrt {n}) \biggr\} \mathop{\longrightarrow}^{(d)} \biggl\{ x \mapsto\frac{1}{\mathcal{Z}_{\sqrt{2}
\beta}}
e^{A_{\sqrt{2} \beta}(x)} e^{-x^2/2} \,dx \biggr\},
\]
where $x \mapsto A_{\beta}(x)$ is a one-parameter family of stationary
processes whose one-point marginal distributions $G_{\beta}$ are the
so-called \textit{crossover distributions}, introduced in \cite{ACQ,SpohnSasa:prl}.
\item Under the intermediate disorder scaling and a diffusive scaling
of space and time, the polymer transition probabilities converge in law
as $n \to\infty$, that is,
\begin{eqnarray*}
&&\biggl\{ (s,y;t,x) \mapsto\frac{\sqrt{n}}{2} \Pl_{n, \beta_n}^{\omega} (
S_{nt} = x \sqrt{n} | S_{ns} = y \sqrt{n} ) \biggr\}\\
&&\qquad
\mathop{\longrightarrow}^{(d)} \frac
{\mathcal{Z}_{\sqrt{2} \beta}(s,y;t,x) \int\mathcal{Z}_{\sqrt{2}
\beta
}(t,x;1,\lambda) \,d \lambda}{\mathcal{Z}_{ \sqrt{2} \beta}}
\end{eqnarray*}
for $0 \leq s < t \leq1$ and $x,y \in\R$. Here $\mathcal{Z}_{\beta
}(s,y;t,x)$ is a random field determined by solutions to the stochastic
heat equation with multiplicative noise, and again it has an explicit
Wiener chaos expansion.
\end{itemize}

As $\beta$ varies, it is believed that the stationary processes
$A_{\beta}(x)$ interpolate between a Gaussian process as $\beta\to0$
and the $\Ai_2$ process as $\beta\to\infty$. Convergence to $\Ai_2$
is currently only known on the level of the one-point marginal
distributions, which were shown in \cite{ACQ,SpohnSasa:prl} to
converge to Tracy--Widom GUE. This interpolation property justifies the
name crossover and emphasizes the importance of the process $A_{\beta}(x)$.

The rest of this paper is organized as follows. In the next section we
give a precise formulation of our main results and sketch the main
ideas behind the proofs. In Section~\ref{sec:notation} we provide some
background material on white noise and stochastic integration, and in
Section~\ref{U_stat_section} we develop the theory of $U$-statistics on a
space--time domain. These theorems form the main technical component of
our paper. In Section~\ref{part_function_convergence_section} we use
the $U$-statistics results to prove Theorem~\ref{Zn_convergence_theorem}.
The proofs of Theorems~\ref{p2p_zn_convergence_theorem} and \ref
{theorem:4_param_convergence} follow in Sections~\ref{sec:rllt} and
\ref
{sec:4_param_field}, although they are very similar to what is done in
Section~\ref{part_function_convergence_section}. The proofs of the
tightness for Theorems~\ref{p2p_zn_convergence_theorem} and~\ref
{theorem:4_param_convergence} are based on standard SPDE arguments
adapted to our situation, and we defer them until the \hyperref[app]{Appendix}.
We end
the main text in Section~\ref{sec:conclusion} with some remarks and
ideas for future work.

%s2 #&#
\section{Formulation of main results}
\label{sec:results}

We begin this section with a brief explanation of why the $n^{-1/4}$
scaling is the appropriate one, and then proceed with precise
formulations of the results and some ideas of the proofs.

%s2.1 #&#
\subsection{Critical scaling}\label{sec2.1}

It is not immediately obvious why $n^{-1/4}$ should be the critical
scaling for intermediate disorder, and there is more than one heuristic
explanation that can be given; see \cite{CLdR}, for example. The
simplest one is in terms of the partition function. Under the
$n^{-1/4}$ scaling it has the form
\[
Z_n^{\omega} \bigl(\beta n^{-{1}/{4}} \bigr) = \Pl \bigl[
\exp \bigl\{ \beta n^{-
{1}/{4}}H_n^{\omega}(S) \bigr\}
\bigr].
\]
Expanding the exponential as a Taylor series and keeping only the terms
up to order $n^{-1/4}$ gives
\[
Z_n^{\omega} \bigl(\beta n^{-{1}/{4}} \bigr) \approx\Pl
\Biggl[ 1 + \beta n^{-{1}/{4}}\sum_{i=1}^n
\omega(i, S_i) \Biggr] = 1 + \beta n^{-{1}/{4}}\sum
_{i=1}^n \sum_{x \in\Z}
\omega(i,x) \Pl(S_i = x).
\]
By computing the variance of the right-hand side it is easily checked
that the $n^{-1/4}$ scaling keeps the random term bounded. In fact, as
the $\omega(i,x)$ variables are i.i.d. with mean zero and variance one,
it is a simple exercise with characteristic functions to show that
\[
\beta n^{-{1}/{4}}\sum_{i=1}^n \sum
_{x \in\Z} \omega(i,x) \Pl (S_i = x)
\mathop{\longrightarrow}^{(d)} N \bigl(0, \sigma^2 \bigr),
\]
where $\sigma^2 = 2 \beta^2 / \sqrt{\pi}$. Hence, up to first-order at
least, the partition function converges in law under the $n^{-1/4}$
scaling. In fact the same is also true of the higher-order terms. One
can simply expand the exponential into a full power series, switch the
expectation on paths with the summation and then analyze each term
individually. For technical reasons, however, it is much easier to make
the $e^{x} \approx1+x$ approximation first and consider instead the
slightly modified partition function
%
%e4 #&#
\begin{equation}
\label{Zbar_def} \symz_n^{\omega} \bigl(\beta n^{-{1}/{4}}
\bigr) = \Pl \Biggl[ \prod_{i=1}^n \bigl(
1 + \beta n^{-{1}/{4}}\omega(i, S_i) \bigr) \Biggr].
\end{equation}
This partition function is the one that was originally introduced and
studied (without any scaling) in the seminal papers \cite
{ImSpen:diffusion,bolthausen:note} for random $\pm1$ environment
variables. The advantage of the $\symz_n^{\omega}$ partition function
is that it
can be more easily analyzed by expanding the product along each path,
leading to
%
%e5 #&#
\begin{equation}
\label{Zbar_expansion1} \symz_n^{\omega} \bigl(\beta n^{-{1}/{4}}
\bigr) = \Pl \Biggl[ 1 + \sum_{k=1}^n
\beta^k n^{-{k}/{4}} \sum_{\vi
\in D_k^n} \prod
_{j=1}^k \omega(\vi_j,
S_{\vi_j}) \Biggr].
\end{equation}
Here $D_k^n$ is the discrete integer simplex
\[
D_k^n = \bigl\{ \vi= (i_1, \ldots,
i_k) \in\N^k \dvtx 1 \leq i_1 < \cdots<
i_k \leq n \bigr\}.
\]
For each $\vi\in D_k^n$ we now average over the possible
configurations of the random walk path at those times. By the Markov
property for simple random walk, the probability of each configuration
is given by the usual product of the transition kernels. Hence
%
%e6 #&#
\begin{eqnarray}
\label{Zbar_expansion2} &&\symz_n^{\omega} \bigl(\beta n^{-{1}/{4}}
\bigr)
\nonumber
\\[-8pt]
\\[-8pt]
\nonumber
&&\qquad= 1 + \sum_{k=1}^n \beta
^k n^{-{k}/{4}} \sum_{\vi
\in
D_k^n} \sum
_{\vx\in\Z^k} \prod_{j=1}^k
\omega(\vi_j, \vx_j) p(\vi_j -
\vi_{j-1}, \vx_j - \vx_{j-1}),
\end{eqnarray}
where $\vi_0 = \vx_0 = 0$, and $p(i,x) = \Pl(S_i = x)$ is the simple
random walk transition kernel. This is the full expansion of $\symz
_n^{\omega}
(\beta n^{-1/4})$ into terms of all orders, and it is possible to
analyze each order individually. The $k = 1$ term we have already shown
converges to a normal random variable. Unfortunately it is not as easy
to write down the limiting distribution of the individual $k > 1$ terms
of the summation, but an explicit form of the limiting variables can
easily be guessed. Very roughly speaking, if one scales space and time
diffusively then the random walk transition probabilities in \eqref
{Zbar_expansion2} approach the transition probabilities for a Brownian
motion, and the environment variables $\omega$ on $\Z_+ \times\Z$
begin to look a white noise on $\R_+ \times\R$. The sums then become
multiple integrals (over free space and ordered time) of the white
noise weighted by the transition kernels for Brownian motion. This type
of multiple stochastic integral is the $k$th order term of a Wiener
chaos expansion.

%s2.2 #&#
\subsection{Main results and ideas of the proof}

In this section we formulate our main results and provide sketches of
the proofs. We believe that these sketches provide sufficient insight
for many readers, and the formal proofs can be found beginning in
Section~\ref{U_stat_section}. The power series expansion of the previous section brings us
to the following theorem:

%th2.1 #&#
\begin{theorem}\label{Zn_convergence_theorem}
We have the following:
\begin{itemize}
\item Assume that the $\omega$ variables have mean zero and variance
one. Then as $n \to\infty$,
\[
\symz_n^{\omega} \bigl(\beta n^{-{1}/{4}} \bigr)
\mathop{\longrightarrow}^{(d)} \mathcal {Z}_{\sqrt{2} \beta}.
\]

\item Assuming that $\lambda(\beta) < \infty$ for $\beta$ small (with
or without the normalizations on the mean and variance), we have the convergence
\[
e^{-n \lambda(\beta n^{-{1}/{4}})} Z_n^{\omega} \bigl(\beta n^{-
{1}/{4}}
\bigr) \mathop{\longrightarrow}^{(d)} \mathcal {Z}_{\sqrt{2} \beta}.
\]

\item The limiting variable $\mathcal{Z}_{\sqrt{2} \beta}$ can be
identified as the sum of multiple stochastic integrals given by
%
%e7 #&#
\begin{equation}
\label{Zn_wiener_chaos} \mathcal{Z}_{\sqrt{2} \beta}:= 1 + \sum
_{k=1}^{\infty} ( \sqrt {2} \beta )^k \int
_{\Delta_k} \int_{\R^k} \prod
_{i=1}^k W(t_i, x_i)
\varrho (t_i - t_{i-1}, x_i -
x_{i-1}) \,dx_i \,dt_i.
\end{equation}
Here $W(t,x)$ is a white noise on $\R_+ \times\R$ with covariance
$\E[ W(t,x) W(s,y) ] = \delta(t-s) \delta(x-y)$, $\Delta_k = \{
0 =
t_0 < t_1 < t_2 < \cdots< t_k \leq1 \}$ is the $k$-dimensional
simplex, $x_i \in\R$ with $x_0 = 0$, and $\varrho$ is the standard
Gaussian heat kernel
\[
\varrho(t,x) = \frac{e^{-x^2/2t}}{\sqrt{2 \pi t}}.
\]
\end{itemize}
\end{theorem}

%re1 #&#
\begin{remark*}
Note that the convergence of $\symz_n^{\omega}$ only requires two
moments for the
environment variables; higher moments are not needed. The requirement
that the variables have mean zero and variance one is only a
normalization condition.
\end{remark*}

%re2 #&#
\begin{remark*}
For the second statement on convergence of $Z_n^{\omega}$ we assumed
that the $\omega$ have exponential moments, but we do not believe that
this is necessary. In the conclusions we discuss our conjecture that
six moments are sufficient for the convergence to go through.
\end{remark*}

%re3 #&#
\begin{remark*}
This theorem is in contrast to what is expected for the strong disorder
case. Observe that
\[
\log Z_n^{\omega} \bigl(\beta n^{-{1}/{4}} \bigr) - n
\lambda \bigl(\beta n^{-{1}/{4}} \bigr)
\]
converges (in law) as $n \to\infty$; hence it is an immediate
corollary that $\chi= 0$ under intermediate disorder.
\end{remark*}

Readers familiar with Gaussian Hilbert spaces will immediately
recognize \eqref{Zn_wiener_chaos} as a Wiener chaos expansion. A good
source for background material on Gaussian Hilbert spaces and Wiener
chaos, and one that we will draw on throughout this work, is \cite
{Janson:GHS}. We include a brief background in Section~\ref{sec:results}. The
distribution of this particular Wiener chaos \eqref{Zn_wiener_chaos}
series is not known, though there are some intriguing conjectures \cite
{calledou:KPZ}. Nonetheless it is still a very concrete expression to
manipulate and study. Using the concept of Wick products (see \cite
{Janson:GHS}), we can rewrite \eqref{Zn_wiener_chaos} as
%
%e8 #&#
\begin{equation}
\label{Zbeta_def} \mathcal{Z}_{\sqrt{2} \beta} = \El_0 \biggl[: \exp:
\biggl\{ \sqrt{2} \beta\int_0^1
W(s,B_s) \,ds \biggr\} \biggr].
\end{equation}
The expectation $\El_0$ is over $1$-dimensional Brownian paths started
at the origin. This shorthand is mostly formal since the integral of
white noise over a Brownian path is not defined on a path-by-path
basis. The procedure that is really indicated by \eqref{Zbeta_def} is
to expand the exponential in a power series and then switch the
expectation over paths with the summation of the series. Formally this
produces the same series as in \eqref{Zn_wiener_chaos}, except that one
uses the Wick exponential $: \exp:$ as a reminder that powers of
integrals should be expanded via the rule
\[
: \biggl( \int_0^1 W(s, B_s) \,ds
\biggr)^k\dvtx = k! \int_{\Delta_k} \prod
_{j=1}^k W(t_j, B_{t_j})
\,dt_j
\]
(recall $\Delta_k$ is the $k$-dimensional simplex). With this in mind
it is easily checked that \eqref{Zn_wiener_chaos} and \eqref{Zbeta_def}
are the same, and \eqref{Zbeta_def} should be viewed simply as
shorthand for the well-defined Wiener chaos \eqref{Zn_wiener_chaos}.
The $\sqrt{2}$ factor in the exponential is a (rather annoying) feature
of the periodicity of simple random walk. It can be seen as a
manifestation of the factor of two in the local limit theorem
%
%e9 #&#
\begin{equation}
\label{llt} p(n,x) = \frac{2}{\sqrt{2 \pi n}} e^{-{x^2}/{(2n)}} + O
\bigl(n^{-{3}/2} \bigr) = \frac{2}{\sqrt{n}} \varrho \bigl(1,xn^{-{1}/2}
\bigr) + O \bigl(n^{-{3}/2} \bigr)
\end{equation}
for $x$ and $n$ of the same parity. Each of the $k$ terms in the
product part of \eqref{Zbar_expansion2} contributes an extra factor of
two in the variance when scaled diffusively, which causes the switch
from $\beta$ to $\sqrt{2} \beta$ in going from the discrete partition
function to the continuum one.

Our proof of Theorem~\ref{Zn_convergence_theorem} essentially follows
the strategy that we have already outlined. First we show that, for
each fixed $k$, the $k$th term in the expansion \eqref{Zbar_expansion2}
converges to the $k$th term of the Wiener chaos \eqref
{Zn_wiener_chaos}. We find it convenient to use the techniques of
\textit{U-statistics} \cite{DynkinMandelbaum:MSI,Gine:st_flour,Janson:GHS}, where such results are the main focus. The general problem
begins with an i.i.d. sequence of real-valued random variables $X_1,
X_2, \ldots,$ and a symmetric function $f \dvtx \R^k \to\R$. One of the
main goals is to find limit theorems for sequences of the form
\[
n^{-k \gamma} \sum_{\vi\in D_k^n} f(X_{\vi_1},
\ldots, X_{\vi_k})
\]
as $n \to\infty$. Here $f$ is thought of as an observable of $k$
variables, and the summation is over all possible random observations
that can be drawn from the set of $n$ variables. A weight function $g \dvtx D_k^n \to\R$ may be added, leading to the study of sums
\[
n^{-k \gamma} \sum_{\vi\in D_k^n} g(\vi)
f(X_{\vi_1}, \ldots, X_{\vi_k}).
\]
This is called an \textit{asymmetric statistic}, and for each $k$ the
terms in \eqref{Zbar_expansion2} have this form. In Section~\ref{U_stat_section} we prove a technical lemma showing that the asymmetric
statistics of~\eqref{Zbar_expansion2} converge to the multiple
stochastic integrals of \eqref{Zn_wiener_chaos}, which is the main step
in the proof.

We will also implicitly make use of an extension of $U$-statistics called
\textit{U-processes} \cite{NolanPollard:Uproc_SLLN,NolanPollard:Uproc_FCLT}. The process is formed by varying the weight
functions $g$ through a given set while keeping the realization of the
random variables fixed. In this paper we will deal with families of
functions $g_x$ indexed by $x \in\R$; an asymmetric $U$-process is then
given by
\[
x \mapsto n^{-k \gamma} \sum_{\vi\in D_k^n}
g_x(\vi) f(X_{\vi_1}, \ldots, X_{\vi_k}).
\]
Limiting $U$-process can then be constructed by taking $n \to\infty$,
and the limits take the form of multiple stochastic integrals of white
noise over space and time, with the noise weighted by the kernels $g_x$.

For our purposes this more general framework is useful for studying the
limit of the \textit{point-to-point} partition function, defined in the
following way
%
%e10 #&#
\begin{equation}
\label{p2p_function} Z_n^{\omega} \bigl(x; \beta n^{-{1}/{4}}
\bigr) = \Pl \bigl[ \exp \bigl\{ \beta n^{-{1}/{4}}H_n^{\omega}(S)
\bigr\} \mathbf{1} \{ S_n = x \} \bigr].
\end{equation}
Through this object we can write the polymer endpoint measure as
%
%e11 #&#
\begin{equation}
\label{endpoint_density} \Pl_{n,\beta_n}^{\omega}(S_n = x) =
\frac{Z_n^{\omega}(x; \beta
n^{-{1}/{4}}
)}{Z_n^{\omega}(\beta n^{-{1}/{4}})}.
\end{equation}
As above, it is more convenient to consider the modified point-to-point
partition function
\begin{eqnarray*}
\symz_n^{\omega} \bigl(x; \beta n^{-{1}/{4}} \bigr)&:=&
\Pl \Biggl[ \prod_{i=1}^n \bigl( 1 + \beta
n^{-{1}/{4}}\omega(i, S_i) \bigr) \mathbf{1} \{ S_n
= x \} \Biggr]
\\
&=& \Pl \Biggl[  \prod_{i=1}^n \bigl( 1
+ \beta n^{-{1}/{4}}\omega(i, S_i) \bigr) \Big| S_n = x
\Biggr] p(n,x).
\end{eqnarray*}
Now the expectation is over walks conditioned to be at $x$ at time $n$,
and this changes the Markov transition kernel that weights the noise.
As was done in \eqref{Zbar_expansion2}, it is easily computed that
%
%e12 #&#
\begin{eqnarray}
\label{p2p_expansion} &&\Pl \Biggl[  \prod_{i=1}^n
\bigl( 1 + \beta n^{-{1}/{4}}\omega(i, S_i) \bigr)\Big |
S_n = x \Biggr]
\nonumber
\\[-8pt]
\\[-8pt]
\nonumber
&&\qquad= 1 + \sum_{k=1}^n
\beta^k n^{-{k}/{4}} \sum_{\vi\in D_k^n} \sum
_{\vx
\in\Z
^k} p_x^n(\vi, \vx)
\prod_{j=1}^k \omega(\vi_j,
\vx_j),
\end{eqnarray}
where $p_x$ is the transition kernel
%
%e13 #&#
\begin{eqnarray}
\label{rw_bridge_kernel} p_x^n(\vi, \vx) &=& \Pl(
S_{\vi_1} = \vx_1, \ldots, S_{\vi_k} = \vx
_k | S_n = x)
\nonumber
\\[-8pt]
\\[-8pt]
\nonumber
& =& \frac{p(n - \vi_k, x - \vx_k)}{p(n,x)} \prod
_{j=1}^k p(\vi _j -
\vi_{j-1}, \vx_j - \vx_{j-1})
\end{eqnarray}
for random walks conditioned to be at position $x$ at time $n$. Under
diffusive scaling this transition kernel converges to the one for
Brownian bridges from zero to a fixed endpoint; the same scaling for
the point-to-point partition function leads to:

%th2.2 #&#
\begin{theorem}\label{p2p_zn_convergence_theorem} The following is true:
\begin{itemize}
\item Under the assumption that the $\omega$ have six moments, with
mean zero and variance one, the process
\[
x \mapsto\Pl \Biggl[  \prod_{i=1}^n
\bigl( 1 + \beta n^{-{1}/{4}}\omega (i, S_i) \bigr) \Big|
S_n = x \sqrt{n} \Biggr]
\]
converges weakly (under the topology of the supremum norm on bounded
continuous functions) to the processs
%
%e14 #&#
\begin{equation}
\label{Zbeta_x_def} x \mapsto e^{A_{\sqrt{2} \beta}(x)}:= \mathbb{E} \biggl[: \exp: \biggl
\{ \sqrt {2} \beta\int_0^1 W(s,
X_s + xs) \,ds \biggr\} \biggr],
\end{equation}
where the expectation $\mathbb{E}$ is over Brownian bridges $X_t$ that
go from zero to zero in time one. See equation \eqref{Zbeta_x_chaos1}
for a formal definition of \eqref{Zbeta_x_def}.

\item Under the assumption that the $\omega$ satisfy $\lambda(\beta) <
\infty$ for $\beta$ sufficiently small, we have
%
%e15 #&#
\begin{equation}
\label{Zbeta_on_parabola} e^{-n \lambda(\beta n^{-1/4})}
 \frac{\sqrt{n}}{2} Z_n^{\omega
}
\bigl(x\sqrt {n}; \beta n^{-{1}/{4}} \bigr) \mathop{\longrightarrow}^{(d)}
e^{A_{\sqrt{2} \beta
}(x)} \frac
{e^{-x^2/2}}{\sqrt{2\pi}},
\end{equation}
where the topology is the supremum norm on bounded continuous functions.

\item Under the latter moment assumptions we also have the random local
limit theorem
%
%e16 #&#
\begin{equation}
\label{random_llt} \biggl\{ x \mapsto\frac{\sqrt{n}}{2} \frac{Z_n^{\omega}(\beta n^{-{1}/{4}}; x\sqrt {n})}{Z_n^{\omega}
(\beta n^{-{1}/{4}})} \biggr\}
\mathop{\longrightarrow}^{(d)} \biggl\{ x \mapsto\frac
{1}{\mathcal{Z}_{\sqrt{2} \beta}} e^{A_{\sqrt{2} \beta}(x)} \rho (1,x)
\biggr\},
\end{equation}
in the same topology.
\end{itemize}
\end{theorem}

%re4 #&#
\begin{remark*}
For the first statement of the theorem the requirement of six moments
is purely technical and only used in the proof of the tightness.
Convergence of finite dimensional distributions to \eqref{Zbeta_x_def}
requires only two moments, and we believe that this is all that should
be required for the tightness. This particular six moments assumption,
however, is \textit{not} related to the conjecture for convergence of
the $Z_n^{\omega}$.
\end{remark*}

%re5 #&#
\begin{remark*}
The discrete processes in this theorem are technically not continuous
functions, but we implicitly assume we are linearly interpolating
between the integer values of $x \sqrt{n}$. To keep the notation simple
we do not write this.
\end{remark*}

The $\beta$-indexed family of processes \eqref{Zbeta_x_def} is
interesting in its own right, and we understand a significant amount
about it. The reasons for writing $A_{\beta}$ in the exponential will
soon be apparent. First note that the expectation \eqref{Zbeta_x_def}
is to be interpreted, as with \eqref{Zbeta_def}, as convenient notation
for a particular Wiener chaos expansion. In this case the transition
kernel for Brownian bridges is used instead of the one for Brownian
motion, hence
%
%e17 #&#
\begin{eqnarray}
\label{Zbeta_x_chaos1}\quad&& e^{A_{\sqrt{2} \beta}(x)}\nonumber\\
&&\qquad = 1 + \sum_{k=0}^{\infty}
(\sqrt{2} \beta)^k \int_{\Delta_k} \int
_{\R^k} \frac{\varrho(1-t_k, x-x_k)}{\varrho(1,x)}
\\
&&\hspace*{114pt}\qquad{}\times\prod_{i=1}^k
W(t_i, x_i) \varrho(t_i -
t_{i-1}, x_i - x_{i-1}) \,dt_i
\,dx_i.\nonumber
\end{eqnarray}
This Wiener chaos expansion is the proper definition of how the process
varies with~$x$, and clearly shows the connection with $U$-processes. As
$x$ varies it is the weight function of the white noise that changes
and this defines the process. However, the probabilistic shorthand
\eqref{Zbeta_x_def}, while only formal, provides a much more intuitive
description than \eqref{Zbeta_x_chaos1}. Observe that \eqref
{Zbeta_x_def} only uses \textit{one} Brownian bridge from zero to zero
and then simply adds the appropriate drift to create a bridge to a
different endpoint. This is, of course, a very simple coupling of
Brownian bridges with the same starting point but different endpoints.
The simplicity is useful, however, and the shorthand \eqref
{Zbeta_x_def} has the advantage of making certain facts immediately
obvious, such as:

%pr2.3 #&#
\begin{proposition}\label{Zbeta_stationary_prop}
The process $x \mapsto A_{\beta}(x)$ is stationary.
\end{proposition}

\begin{pf}
Fix $\theta\in\R$, and define $W_{\theta}(t,x):= W(t, x + \theta
t)$. As the linear transformation $(t,x) \mapsto(t, x + \theta t)$ has
determinant $1$, we immediately have that $W_{\theta}$ is also a white
noise on $\R_+ \times\R$ with the same covariance function as $W$. Then
\begin{eqnarray*}
x \mapsto e^{A_{\beta}(x + \theta)} &=& \mathbb{E} \biggl[: \exp: \biggl\{ \beta \int
_0^1 W(s, X_s + xs + \theta s) \,ds
\biggr\} \biggr]
\\
&=& \mathbb{E} \biggl[: \exp: \biggl\{ \beta\int_0^1
W_{\theta}(s, X_s + xs) \,ds \biggr\} \biggr]
\end{eqnarray*}
has the same law as the original process.
\end{pf}

Combining Proposition~\ref{Zbeta_stationary_prop} with Theorem~\ref
{p2p_zn_convergence_theorem} produces an analogue of a well-known
result from the polynuclear growth model \cite
{PrahoferSpohn:airy_process}. Taking logarithms of \eqref
{Zbeta_on_parabola} leads to
%
%e18 #&#
\begin{eqnarray}
\label{logthing} &&\log Z_n^{\omega} \bigl( x\sqrt{n};\beta
n^{-{1}/{4}} \bigr) - n\lambda \bigl(\beta n^{-1/4} \bigr)+
\frac
{1}{2} \log(n/4)
\nonumber
\\[-8pt]
\\[-8pt]
\nonumber
&&\qquad \to A_{\sqrt{2}\beta}(x) - \frac{x^2}{2} - \log
\sqrt {2 \pi}.
\end{eqnarray}
The right-hand side is a stationary process around a parabola, which is
similar to what was first encountered for the PNG droplet \cite
{PrahoferSpohn:airy_process}. There the stationary process was
Airy$_2$, and it was shown that it has the Tracy--Widom distribution
$F_{\mathrm{GUE}}$ as its one-point marginal distribution. On the level of
one-point marginal distributions there is a much stronger connection
between $A_{\beta}$ and $\Ai_2$ (although nothing yet has been
rigorously proved on the process level; see \cite{ACQ} for a
conjecture). Recent results derived by Amir, Corwin and Quastel \cite
{ACQ}, and independently by Spohn and Sasamoto \cite{SpohnSasa:KPZ,SpohnSasa:prl}, give a formula for the one-point marginal distribution
of the process $A_{\beta}$ and how this distribution scales as $\beta$
goes to zero or infinity, specifically as $\beta\to\infty$ the
one-point marginals of $A_{\beta}$ converge to those of $\Ai_2$:

%pr2.4 #&#
\begin{proposition}[(\cite{ACQ})]\label{ACQ_theorem}
For $\beta> 0$ and $x \in\R$,
%
%e19 #&#
\begin{eqnarray}
\label{df} G_{\beta}(s)&:=& \P \bigl( A_{\sqrt{2} \beta}(x) + 2
\beta^4/3 \leq s \bigr)
\nonumber
\\[-8pt]
\\[-8pt]
\nonumber
&=& 1 - \int e^{-e^{-r}} f \bigl( s - \log
\bigl(32 \pi\beta^4 \bigr)/2 - r \bigr) \,dr,
\end{eqnarray}
with
\[
f(r) = \kappa_{\beta}^{-1} \det(I - K_{\sigma_{\beta}}) \mathrm
{tr} \bigl((I - K_{\sigma_{\beta}})^{-1} P_{\Ai} \bigr),
\]
where $\kappa_{\beta} = 2 \beta^{4/3}$, and $K_{{\sigma}_{\beta}}$ and
$P_{\Ai}$ are operators acting on $L^2(\kappa_{\beta}^{-1} r, \infty)$
given by their kernels
\begin{eqnarray*}
P_{\Ai}(x,y) &=& \Aifunc(x) \Aifunc(y),
\\
K_{\sigma_{\beta}}(x,y) &=& \operatorname{P.V.} \int\sigma_{\beta}(t) \Aifunc
(x+t) \Aifunc(y+t) \,dt,
\end{eqnarray*}
with $\sigma_{\beta}(t) = (1 -e^{-\kappa_{\beta}t})^{-1}$. Moreover,
the distribution functions $G_{\beta}$ satisfy the asymptotic relations
%
%e20 #&#
\begin{equation}
\label{gueasymp} G_{\beta} \bigl( 2^{{4}/3} \beta^{{4}/3} s
\bigr) \mathop{\longrightarrow}^{\beta\to \infty} F_{\mathrm{GUE}} \bigl( 2^{{1}/3} s
\bigr)
\end{equation}
and
\[
G_{\beta} \bigl( 2^{{1}/2} \pi^{{1}/4} \beta s \bigr)
\mathop{\longrightarrow}^{\beta \to 0} \int_{-\infty}^s
\frac{e^{-x^2/2}}{\sqrt{2\pi}} \,dx.
\]
\end{proposition}

This result is derived using steepest descent analysis on a
Tracy--Widom formula~\cite{TracyWidom:asep_fredholm_represent,TracyWidom:asep_integral_formulas,TracyWidom:asep_step_condition} for
the Asymmetric Simple Exclusion Process (ASEP). The exact form of the
$G_{\beta}$ distributions are not used in this paper, but the
asymptotics are important and suggest how the polymer scalings behave
as $\beta\to\infty$.

In particular, combining (\ref{logthing}), (\ref{df}) and (\ref
{gueasymp}), we have
the following corollary, which is the first general result for
Tracy--Widom asymptotics for polymers at nonzero temperature:

%co2.5 #&#
\begin{corollary}[(Weak universality for directed random polymers in
$1+1$ dimensions)]
Assume that $\omega$ are i.i.d. with $\lambda(\beta) < \infty$ for
$\beta$ sufficiently small. Then
as $n\to\infty$ followed by $\beta\to\infty$,
\[
\frac{\log Z_n^{\omega}(0; \beta n^{-{1}/{4}}) - n\lambda(\beta
n^{-1/4}) +\log
\sqrt {\pi n/2} +2\beta^4/3}{2\beta^{4/3}} \mathop{\longrightarrow}^{(d)} F_{ \mathrm{GUE}}.
\]
\end{corollary}

In the concluding remarks of Section~\ref{sec:conclusion} we discuss
our conjecture that this statement holds whenever the $\omega$ have six
moments, and is false otherwise. The corollary above should be compared
to the following conjecture (see \cite{MR2335697} for the $5$ moment
assumption):

%co2.6 #&#
\begin{conjecture}[(Strong universality for directed random polymers in
$1+1$ dimensions)]
Assume that $\omega$ are i.i.d. with $5$ moments. Then there are
constants $c = c(\beta)$ and $\sigma= \sigma(\beta)$ such that as
$n\to
\infty$,
\[
\frac{\log Z_n^{\omega}(0; \beta) - c(\beta)n}{\sigma(\beta) n^{1/3}}
\mathop{\longrightarrow}^{(d)} F_{\mathrm{GUE}}
\]
for all $\beta> 0$.
\end{conjecture}

Finally we consider the transition probabilities for the polymer. It is
worth emphasizing that, given the environment field $\omega$, the
polymer measure $\Pl_{n, \beta}^{\omega}$ is Markov, and therefore it
is uniquely determined by its transition probabilities. These
transition probabilities are functions of the environment, and the
inhomogeneous nature of the environment means that the probabilities
are inhomogeneous in space and time. For $0 \leq m < k \leq n$ and $x,y
\in\Z$ we define the four-parameter field $Z^{\omega}(m,y;k,x;\beta
)$ by
%
%e21 #&#
\begin{equation}
\label{defn:4_param_field} Z^{\omega}(m,y;k,x;\beta) = \Pl \Biggl[  \exp \Biggl\{
\beta \sum_{i=m+1}^{k} \omega(i,
S_i) \Biggr\} \mathbf{1} \{ S_k = x \} \Big|
S_m = y \Biggr],
\end{equation}
which is a point-to-point partition function for a polymer starting at
position $y$ at time $m$. For the point-to-line versions we introduce
the notation
%
%e22 #&#
\begin{equation}
\label{defn:p2l_param_field} Z^{\omega}(m,y;k,*;\beta) = \sum
_{x \in\Z} Z^{\omega
}(m,y;k,x;\beta).
\end{equation}
It is then straightforward to verify that for $0 \leq m < k \leq n$ the
polymer has transition probabilities given by
%
%e23 #&#
\begin{eqnarray}
\label{polymer_trans_prob} &&\Pl_{n, \beta}^{\omega}(S_{i+1} =
S_i \pm1 | S_1, \ldots, S_i)
\nonumber
\\[-8pt]
\\[-8pt]
\nonumber
&&\qquad =
\frac{1}{2} e^{\beta\omega(i+1, S_i \pm1)} \frac{Z^{\omega}(i+1, S_i
\pm1; n, *; \beta)}{Z^{\omega}(i, S_i; n, *; \beta)}.
\end{eqnarray}
The polymer is clearly Markov since the equation on the right only
depends on~$S_i$. More importantly though, this equation shows that the
four-parameter field uniquely determines the polymer measure. Our final
theorem describes its scaling limit as $n \to\infty$. As before we
initially work with the modified partition function
%
%e24 #&#
\begin{equation}
\label{defn:4_param_modified} \symz^{\omega}(m,y;k,x;\beta) = \Pl \Biggl[  \prod
_{i=m+1}^k \bigl( 1 + \beta \omega(i,
S_i) \bigr) \mathbf{1} \{ S_k = x \} \Big| S_m
= y \Biggr]
\end{equation}
and then transfer the results to the usual exponential form.

%th2.7 #&#
\begin{theorem}\label{theorem:4_param_convergence} The following is true:
\begin{itemize}
\item Assuming that the $\omega$ have six moments with mean zero and
variance one, the fields
\[
(s,y;t,x) \to\frac{\sqrt{n}}{2} \symz^{\omega} \bigl(ns, y \sqrt{n}; nt, x
\sqrt{n}; \beta n^{-1/4} \bigr)
\]
converge weakly as $n \to\infty$ to a random field $\mathcal
{Z}_{\sqrt
{2} \beta}(s,y;t,x)$. The topology is the sup norm on bounded
continuous functions with the domain $\{(s,y;t,x) \dvtx 0 \leq s < t \leq
1, x, y \in\R\}$.

\item Assuming that $\lambda(\beta) < \infty$ for small $\beta$,
the fields
\[
(s,y;t,x) \mapsto\frac{\sqrt{n}}{2} e^{-n(t-s) \lambda(\beta
n^{-1/4})} Z^{\omega} \bigl(ns, y
\sqrt{n}; nt, x \sqrt{n}; \beta n^{-1/4} \bigr)
\]
converge weakly to the same limit, under the same topology.

\item The limiting field $\mathcal{Z}_{\beta}(s,y;t,x)$ has the chaos
representation
\begin{eqnarray*}
\mathcal{Z}_{\beta}(s,y;t,x) &=& \varrho(s,y;t,x)
\\
&&{} + \sum_{k=1}^{\infty} \beta^k
\int_{\Delta_k(s,t]} \int_{\R^k} \prod
_{i=1}^k W(t_i, x_i)\\
&&\hspace*{85pt}\qquad{}\times
\varrho(t_i - t_{i-1}, x_i -
x_{i-1}) \\
&&\hspace*{85pt}\qquad{}\times\varrho(t - t_k, x - x_k)
\,dx_i \,dt_i,
\end{eqnarray*}
with $\Delta_k(s,t] = \{ s = t_0 < t_1 < \cdots< t_k < t \}$, and $x_i
\in\R$ with $x_0 = y$.
\end{itemize}
\end{theorem}

%re6 #&#
\begin{remark*}
As in the remarks after Theorem~\ref{p2p_zn_convergence_theorem}, the
six moments assumption for $\symz_n^{\omega}$ is only used in the
tightness, but
we do not believe it to be necessary. Convergence of the finite
dimensional distributions of $\symz_n^{\omega}$ goes through with
only two
moments. Also, the discrete fields are defined only on the mesh where
$(ns, y \sqrt{n}; nt, x \sqrt{n})$ take integer values, but we use a
linear interpolation scheme to extend it to the whole space. A
particular method is outlined in the \hyperref[app]{Appendix}.
\end{remark*}

%re7 #&#
\begin{remark*} As is discussed in the companion paper \cite{AKQ:CRP},
the four-parameter field is the chaos solution to the stochastic
partial differential equation
\[
\partial_t \mathcal{Z}_{\beta} = \tfrac{1}2
\partial_{xx} \mathcal {Z}_{\beta} + \beta W
\mathcal{Z}_{\beta},\qquad \mathcal{Z}_{\beta}(s,y;s,x) =
\delta_0(x-y).
\]
This is the stochastic heat equation with multiplicative noise. The
logarithm of this field is the Hopf--Cole solution of the so-called KPZ
equation \cite{KPZ,BG,ACQ}. The continuum analogue of the field
\eqref
{defn:p2l_param_field} is the point-to-line partition function defined by
\[
\mathcal{Z}_{\beta}(s,y;t,*) = \int\mathcal{Z}_{\beta}(s,y;t,x)
\,dx.
\]
\end{remark*}

%re8 #&#
\begin{remark*}
From equation \eqref{polymer_trans_prob} it can easily be derived that
the multi-step polymer transition probabilities are
\[
\Pl_{n, \beta}^{\omega}(S_k = x | S_m = y) =
\frac{Z^{\omega}(m,y; k,x
;\beta) Z^{\omega}(k,x;n,*;\beta)}{Z^{\omega}(m, y; n, *; \beta)}.
\]
Combining this with Theorems \ref{Zn_convergence_theorem} and \ref
{p2p_zn_convergence_theorem} leads to the statement in the \hyperref[sec1]{Introduction} that
\[
\frac{\sqrt{n}}{2} \Pl_{n, \beta_n}^{\omega} ( S_{nt} = x
\sqrt {n} | S_{ns} = y \sqrt{n} ) \mathop{\longrightarrow}^{(d)} \frac{\mathcal{Z}_{\sqrt{2}
\beta}(s,y;t,x) \mathcal{Z}_{\sqrt{2} \beta}(t,x;1,*)}{\mathcal{Z}_{
\sqrt{2} \beta}}.
\]
\end{remark*}

%s3 #&#
\section{Wiener chaos}
\label{sec:notation}

%s3.1 #&#
\subsection{Brownian motion and simple random walk}
\label{subsec:brownian_terms}
Throughout we let $S_n$ be a simple random walk on $\Z$ and $B_t$ be a
standard Brownian motion on $\R$. For $i \in\N$ or $t \geq0$ and $x
\in\R$ let
\[
p(i,x) = \Pl(S_i = x), \qquad \varrho(t,x) = \frac{e^{-x^2/2t}}{\sqrt{2
\pi t}}.
\]
We will make heavy use of the joint probability densities of both
simple random walk and Brownian motion, for which we introduce the notation
\[
p_k(\vi, \vx) = \Pl( S_{\vi_1} = \vx_1, \ldots,
S_{\vi_k} = \vx _k ) = \prod_{j=1}^k
p(\vi_j - \vi_{j-1}, \vx_j -
\vx_{j-1})
\]
for $\vi\in D_k^n$, $\vx\in\R^k$. Here $D_k^n$ is the integer simplex
\[
D_k^n = \bigl\{ \vi\in[ n ]^k \dvtx 1 \leq
\vi_1 < \vi_2 < \cdots< \vi_k \leq n \bigr\},
\]
with $[n]:= \{1, 2, \ldots, n \}$. For $\vi\in\Z^k$ we sometimes
write $|\vi|$ as shorthand for the length $k$ of the vector.

The parity property of simple random walk, that it is only on the even
integers at even times and the odd integers at odd times, plays a role
in much of the technical analysis. We write $i \leftrightarrow x$ if
$i$ and $x$ are of the same parity, and for $\vi\in D_k^n, \vx\in\Z
^k$ we write $\vi\leftrightarrow\vx$ if $\vi_j \leftrightarrow\vx_j$
for $1 \leq j \leq k$. Given $x \in\R$ and $i \in\N$ we define
$[ x ]_i$ to be the closest element of $\Z$ to $x$ such that $i
\leftrightarrow x$. For a point $\vx\in\R^k$ and $\vi\in D_k^n$
define $[ \vx ]_{\vi} \in\Z^k$ by $( [ \vx ]_{\vi} )_k =
[ x_k ]_{i_k}$. It will often be useful to extend $p$ to all of $\R^k$,
so we also define
\[
\overline{p}_k(\vi, \vx) = 2^{-k} p_k \bigl(
\vi, [ \vx ]_{\vi} \bigr).
\]
The $2^{-|\vi|}$ factor normalizes $\overline{p}(\vi, \cdot)$ to be a
probability measure on $\R^k$. It is useful to observe that the
$\overline{p}(\vi, \vx)$ are the finite dimensional distributions for
the random walk $X_n = S_n + U_n$, where $\{ U_{i} \}_{i \geq1}$ is a
sequence of independent random variables uniformly distributed on $(-1,1)$.

For Brownian motion we use the analogous notation
\[
\varrho_k(\vt, \vx) = \prod_{j=1}^k
\varrho(\vt_j - \vt_{j-1}, \vx_j -
\vx_{j-1})
\]
for $\vt$ in the simplex
\[
\Delta_k = \bigl\{ \vt\in\R^k \dvtx 0 \leq
\vt_1 < \vt_2 < \cdots<\vt_k \leq1 \bigr\}.
\]

%s3.2 #&#
\subsection{\texorpdfstring{White noise and stochastic integration on $[0,1]\times\R$}
{White noise and stochastic integration on [0,1] x R}}
\label{wn_si_section}

\newcommand{\envW}{\mathbb{Q}}

In this section we briefly recall the elementary theory of white noise
and stochastic integration on the particular measure space $L^2([0,1]
\times\R, \B, dt \,dx)$. Here $\B$ is the $\sigma$-algebra of Borel
subsets, and $dt \,dx$ denotes Lebesgue measure on the space. We let $\B
_f$ be the subset of $\B$ consisting of sets of finite Lebesgue
measure. Observe that $\B= \sigma(\B_f)$ because Lebesgue measure is
$\sigma$-finite on the given space.

A white noise on $[0,1] \times\R$ is a collection of mean zero
Gaussian random variables defined on a common probability space
$(\Omega
_W, \F_W, \envW)$ and indexed by $\B_f$
\[
W = \bigl\{ W(A) \dvtx A \in\B_f \bigr\}.
\]
This means that every finite collection of the form $(W(A_1), \ldots,
W(A_k))$ has a $k$-dimensional Gaussian distribution, with mean zero
and covariance structure given by
\[
\E \bigl[ W(A)W(B) \bigr] = | A \cap B |.
\]
In particular if $A$ and $B$ are disjoint then $W(A)$ and $W(B)$ are
independent.

For $g \in L^2([0,1] \times\R, \B, dt \,dx)$ the stochastic integral
\[
I_1(g) = \int_0^1 \int g(t,x)
W(dt \,dx)
\]
is constructed by first defining $I_1$ on simple functions via
\[
I_1 \Biggl( \sum_{i=1}^n
\alpha_i \textbf{1}_{A_i} \Biggr) = \sum
_{i=1}^n \alpha_i
W(A_i),
\]
where $A_i \in\B_f$, and then using the density of simple functions in
$L^2([0,1] \times\R)$ and the completeness of $L^2(\Omega_W, \F_W,
\envW)$ to define $I_1(g)$. The construction shows that $I_1$ is linear
in the sense that for all $\alpha_1, \ldots, \alpha_n \in\R$ and $g_1,
\ldots, g_n \in L^2([0,1] \times\R)$ we have, with probability one,
\[
I_1 \Biggl( \sum_{i=1}^n
\alpha_i g_i \Biggr) = \sum
_{i=1}^n \alpha_i
I_1(g_i).
\]
For each $g$ we have that $I_1(g) \sim N(0, \Vert g\Vert _{L^2}^2)$, and
moreover $I_1$ preserves the Hilbert space structure of $L^2([0,1]
\times\R)$,
\[
\E \bigl[ I_1(g) I_1(h) \bigr] = \int
_0^1 \int g(t,x) h(t,x) \,dt \,dx.
\]

Now we define multiple stochastic integrals on $L^2([0,1]^k \times\R
^k)$ for $k > 1$. For notation we use
\[
I_k(g) = \int_{[0,1]^k} \int_{\R^k}
g(\vt, \vx) W^{\otimes k}(d\vt \,d\vx)
\]
for $g \in L^2([0,1]^k \times\R^k)$. The construction is similar to
the $k=1$ case except that mild care must be given to integration along
the ``diagonals'' of the space. Moreover, the stochastic integral is
only truly defined for symmetric functions on $[0,1]^k \times\R^k$.
Here $g$ is symmetric if $g(\vt, \vx) = g(\pi\vt, \pi\vx)$ for all
$(\vt, \vx) \in[0,1]^k \times\R^k$ and $\pi\in S_k$, the group of
permutations on $\{1, 2, \ldots, k \}$. The permutations act on vectors
in the obvious way: $\pi\vt= (\vt_{\pi1}, \ldots, \vt_{\pi k})$ and
$\pi\vx$ defined similarly. We let $L^2_S([0,1]^k \times\R^k)$ denote
the subspace of symmetric functions in $L^2$. As in the $k=1$ case it
is enough to define $I_k$ on a dense subset of $L_S^2$ from which it
can be linearly extended (in a unique way) to the entire space. One
dense subset is the functions of the form
%
%e25 #&#
\begin{equation}
\label{dense_g_products} g(\vt, \vx) = \sum_{\pi\in S_k} \prod
_{j=1}^k \mathbf{1} \bigl\{ (\vt
_{\pi j}, \vx_{\pi j}) \in A_{j} \bigr\},
\end{equation}
where the $A_j, j=1, \ldots, k$ are disjoint subsets of $[0,1] \times
\R
$. We define
%
%e26 #&#
\begin{equation}
\label{multiple_integral_def} I_k(g) = k! \prod_{j=1}^k
W(A_j).
\end{equation}
It is standard to show that there exists a unique linear extension of
$I_k$ from functions of the form \eqref{dense_g_products} onto $L^2_S$
such that each $I_k(g)$ is a mean zero random variable with variance
$\Vert g\Vert _{L^2}^2$, and the covariance structure is
\[
\E \bigl[ I_k(g) I_k(h) \bigr] = \langle g,h
\rangle_{L^2([0,1]^k \times\R^k)}.
\]
Moreover, relation \eqref{multiple_integral_def} can be extended to
show that if $g_1, \ldots, g_k \in L^2([0,1] \times\R)$ are all
orthogonal to each other, then
%
%e27 #&#
\begin{equation}
\label{prod_of_disjoint_g} I_k \Biggl( \sum_{\pi\in S_k}
\prod_{j=1}^k g_j(
\vt_{\pi j}, \vx_{\pi
j}) \Biggr) = k! \prod
_{j=1}^k I_1(g_j).
\end{equation}
Finally, we also adopt the convention that $I_k$ extends to
nonsymmetric functions $g \in L^2([0,1]^k \times\R^k)$ via
symmetrization. We define
\[
I_k(g) = I_k(\sym g),
\]
where
\[
\sym g(\vt,\vx) = \frac{1}{k!} \sum_{\pi\in S_k} g(
\pi\vt, \pi \vx).
\]

%re9 #&#
\begin{remark*}
Suppose that $g \dvtx \Delta_k \times\R^k \to\R$ instead. We extend it to
a function on $[0,1]^k \times\R^k$ by defining it to be zero for $\vt
\notin\Delta_k$. Then
\[
I_k(g) = \int_{[0,1]^k} \int_{\R^k}
\sym g(\vt, \vx) W^{\otimes k}(d \vt\, d \vx) = k! \int_{\Delta_k}
\int_{\R^k} g(\vt, \vx) W^{\otimes
k}(d \vt\, d \vx)
\]
since symmetrizing simply ``copies'' the functions into the $k!$
permutations of $\Delta_k$ that make up $[0,1]^k$, ignoring the
diagonals which do not affect the stochastic integral anyways.
\end{remark*}

%s3.3 #&#
\subsection{\texorpdfstring{Wiener chaos on $[0,1]\times\R$}{Wiener chaos on [0,1] x R}}

In the context of this paper Wiener chaos may be regarded as a way of
representing random variables as infinite sums of multiple stochastic
integrals. For every random variable $X \in L^2(\Omega_W, \F_W, \envW
)$, the Wiener chaos decomposition states that there is a unique
sequence of symmetric functions $g_k \in L_S^2([0,1]^k \times\R^k), k
\geq1$, such that
\[
X = \sum_{k=0}^{\infty} I_k(g_k).
\]
Here $g_0$ is simply a constant and $I_0(g_0) = g_0$. In fact as the $k
\geq1$ terms of the chaos series are all mean zero, $g_0$ must be the
mean of $X$. Moreover, by the orthogonality of $I_{k_1}(g_1)$ and
$I_{k_2}(g_2)$ for $k_1 \neq k_2$ we have the relation
\[
\E \bigl[ X^2 \bigr] = \sum_{k=0}^{\infty}
\Vert g_k\Vert _{L^2([0,1]^k \times\R^k)}^2.
\]
The situation works in reverse also. Given any element of the symmetric
Fock space over $L^2([0,1] \times\R)$, that is,
\[
g = (g_0, g_1, g_2, \ldots) \in\bigoplus
_{k=0}^{\infty} L_S^2
\bigl([0,1]^k \times\R^k \bigr),
\]
the map $I \dvtx \bigoplus_{k=0}^{\infty} L_S^2([0,1]^k \times\R^k) \to
L^2(\Omega_W, \F_W, \envW)$ defined by $I(g) =  \sum_{k \geq0} I_k(g_k)$
is an isometry. The norm on the Fock space is
\[
\Vert g\Vert _F^2 = \sum_{k=0}^{\infty}
\Vert g_k\Vert _{L^2([0,1]^k \times\R^k)}^2.
\]

%s3.4 #&#
\subsection{Wiener chaos for Brownian transition probabilities}
\label{brownian_chaos_section}

The Brownian transition probabilities are easily shown to define an
element of the Fock space $\bigoplus_{k \geq0} L^2(\Delta_k \times\R
^k)$. Define
\[
\bolds{\varrho}(\beta) = \bigl(1, \beta\varrho_1,
\beta^2 \varrho_2, \ldots \bigr),
\]
where the $\varrho_k$ are defined in Section~\ref{sec2.1}. For all $\beta\in
\R
$ it is easily computed that $\bolds{\varrho}(\beta)$ belongs to the
Fock space, since
\[
\varrho_k(\vt, \vx)^2 = \varrho_k(\vt,
\sqrt{2} \vx) \prod_{j=1}^k
\frac
{1}{\sqrt{2 \pi(\vt_j - \vt_{j-1})}},
\]
and hence
\begin{eqnarray*}
\int_{\Delta_k} \int_{\R^k}
\varrho_k(\vt, \vx)^2 \,d \vx\, d \vt&=& (4 \pi
)^{-{k}/{2}} \int_{\Delta_k} \prod
_{j=1}^k \frac{1}{\sqrt{\vt
_j -
\vt_{j-1}}} \,d\vt
\\
&=& (4 \pi)^{-{k}/{2}} B \biggl( \frac{1}{2}, \frac{1}{2},
\ldots, \frac
{1}{2}, 1 \biggr)
=\frac{1}{2^k \Gamma( ({k}/{2}) + 1 )}.
\end{eqnarray*}
The second equality comes from recognizing that the integrand is the
density of the Dirichlet distribution, for which the beta function $B$
is the normalizing constant. Recalling that the beta function is the
ratio of Gamma functions \cite{AbSteg} produces the final expression,
and the extremely fast decay of this expression in $k$ clearly shows
that $\Vert \bolds{\varrho}(\beta)\Vert _F^2 < \infty$ for all $\beta\in
\R$.
This gives the following:

%le3.1 #&#
\begin{lemma}
The Wiener chaos $\mathcal{Z}_{\beta}$ has the representation
$\mathcal
{Z}_{\beta} = I(\bolds{\varrho}(\beta))$.
\end{lemma}

For fixed $y \in\R$ we also define
\[
\varrho_k(\vt, \vx;y) = \varrho_k(\vt, \vx) \varrho(1-
\vt_k, y - \vx_k)
\]
for which it can be computed in a similar manner as above that
\begin{eqnarray*}
&&\int_{\Delta_k} \int_{\R^k}
\varrho_k(\vt, \vx;y)^2 \,d \vx\, d \vt\\
&&\qquad = \frac
{e^{-y^2}}{2^{k + 1/2} \Gamma( {(k+1)}/{2} )}.
\end{eqnarray*}
We define $\bolds{\varrho}(\beta;y) = (1, \beta\varrho_1(\cdot,
\cdot
;y), \beta^2 \varrho_2(\cdot, \cdot; y), \ldots)$, and then
clearly\break
$\sup_y \Vert \varrho(\beta;  y)\Vert _F^2 < \infty$ uniformly in $y$, for all
$\beta\in\R$.

%s4 #&#
\section{$U$-Statistics}
\label{U_stat_section}

In this section we prove the main technical theorem for convergence in
law of the partition functions. The results presented here are standard
in the theory of $U$-statistics. We are mostly interested in sums of the form
%
%e28 #&#
\begin{equation}
\label{generic_ustat_sum} \sum_{\vi\in D_k^n} \mathop{\sum
_{\vx\in\Z^k }}_{ \vi
\leftrightarrow\vx} g(\vi, \vx) \omega(\vi, \vx)
\end{equation}
for some weight function $g$; this particular type of sum is generally
referred to as a weighted or asymmetric $U$-statistic.

%s4.1 #&#
\subsection{$U$-Statistics for space--time random environments}

It will actually be more efficient to slightly generalize our results
to sums over unordered $\vi$, that is, $\vi\in E_k^n$ where
\[
E_k^n = \bigl\{ \vi\in[n]^k \dvtx
\vi_j \neq\vi_l \textrm{ for } j \neq l \bigr\}.
\]
Furthermore the theory will be easier to work with when the weight
function is extracted from an $L^2$ function on $[0,1]^k \times\R^k$.
We first discretize such a function by replacing it by its average on
rectangles; we use rectangles of the form
\[
\mathcal{R}_k^n:= \biggl\{ \biggl( \frac{\vi-\mathbf{1}}{n},
\frac{\vi}{n} \biggr] \times\biggl( \frac{\vx-\mathbf{1}}{\sqrt{n}}, \frac{\vx+\mathbf
{1}}{\sqrt {n}} \biggr] \dvtx \vi\in
D_k^n, \vx\in\Z^k, \vi\leftrightarrow\vx
\biggr\},
\]
with $\mathbf{1}$ being the vector of all ones. Observe that $|R| = 2^k
n^{-{3k}/{2}}$ for $R \in\mathcal{R}_k^n$. For $g \in L^2([0,1]^k
\times\R^k)$ define $\overline{g}_n$ by
\[
\overline{g}_n(\vt, \vx) = \frac{1}{| R |} \int
_R g\qquad (\vt, \vx ) \in R \in\mathcal{R}_n.
\]
In probabilistic terms $\overline{g}_n$ is simply the conditional
expectation of $g$ onto the rectangles of $\mathcal{R}_k^n$. Now we
define weighted $U$-statistics via
%
%e29 #&#
\begin{equation}
\label{sk_def} \S_k^n(g) = 2^{{k}/{2}} \sum
_{\vi\in E_k^n} \sum_{\vx\in\Z^k}
\overline{g}_n \biggl( \frac{\vi}{n}, \frac{\vx}{\sqrt{n}} \biggr)
\mathbf {1} \{ \vi \leftrightarrow\vx \} \omega(\vi, \vx).
\end{equation}

%re10 #&#
\begin{remark*}
Observe that in the space direction the rectangles have length $2/\sqrt {n}$ rather than $1/\sqrt{n}$. This is one way of dealing with the
periodicity issue of simple random walk.
\end{remark*}

From this definition we have the following:

%le4.1 #&#
\begin{lemma}\label{sk_lemma}
The map $\S_k^n$ is linear in the sense that for all $\alpha_1,
\ldots,
\alpha_m \in\R$ and $g_1, \ldots, g_m \in L^2([0,1]^k \times\R
^k)$ we have
\[
\sum_{l=1}^m \alpha_l
\S_k^n(g_l) = \S_k^n
\Biggl( \sum_{l=1}^m \alpha
_l g_l \Biggr),
\]
with probability one. For all $k$ the variables $\S_k^n(g)$ are mean
zero, and for $k_1 \neq k_2$ and $g_i \in L^2([0,1]^{k_i} \times\R
^{k_i})$ we have
\[
Q \bigl[ \S_{k_1}^n(g_1) \S_{k_2}^n(g_2)
\bigr] = 0.
\]
For $k_1 = k_2 = k$ we have
\[
Q \bigl[ \S_k^n(g)^2 \bigr] \leq
n^{{3k}/{2}} \Vert \overline{g}_n \Vert _{L^2([0,1]^k \times\R^k)}^2
\leq n^{{3k}/{2}} \Vert g \Vert _{L^2([0,1]^k \times\R^k).}^2
\]
\end{lemma}

\begin{pf}
The linearity and mean zero properties are obvious. For the covariance
relation for $k_1 \neq k_2$, simply observe that if $\vi\in E_{k_1}^n,
\vx\in\Z^{k_1}, \vi' \in E_{k_2}^n, \vx' \in\Z^{k_1}$, then
\[
Q \Biggl[ \prod_{j_1=1}^{k_1} \omega(
\vi_{j_1}, \vx_{j_1}) \prod_{j_2 =
1}^{k_2}
\omega \bigl(\vi_{j_2}', \vx_{j_2}'
\bigr) \Biggr] = 0
\]
since there is necessarily one $\omega$ term that is distinct from all
others, and its independence from the rest implies zero expectation. In
the $k_1 = k_2 = k$ case observe that
\[
Q \Biggl[ \prod_{j=1}^{k} \omega(
\vi_{j}, \vx_{j}) \prod_{j = 1}^{k}
\omega \bigl(\vi_{j}', \vx_{j}'
\bigr) \Biggr] = \mathbf{1} \bigl\{ \vi= \vi', \vx=
\vx' \bigr\}.
\]
Hence
\begin{eqnarray*}
Q \bigl[ \S_k^n(g)^2 \bigr] &=&
2^k \sum_{\vi\in E_k^n} \sum
_{\vx\in\Z^k} \overline{g}_n \biggl( \frac{\vi}{n},
\frac{\vx}{\sqrt{n}} \biggr)^2
\\
& \leq&2^k \sum_{\vi\in[n]^k} \sum
_{\vx\in\Z^k} \frac{\mathbf
{1} \{ \vi \leftrightarrow\vx \}}{|R|} \int_R g(
\vt, \vy)^2 \,d \vt\, d \vy
\\
&=& n^{{3k}/{2}} \int_{[0,1]^k} \int_{\R^k}
g(\vt, \vy)^2 \,d \vt\, d \vy.
\end{eqnarray*}
The inequality is an application of the Cauchy--Schwarz lemma, or
simply from the fact that $\overline{g}_n$ is a conditional expectation
of $g$.
\end{pf}

We next state a standard weak convergence result that we will make
repeated use of. A proof can be found in \cite{billingsley}, Chapter~1, Theorem~4.2.

%le4.2 #&#
\begin{lemma}\label{wk_lemma}
Let $Y_k^n, Y_k, Y^n, Y$ be real-valued random variables, and assume
that for each fixed $n$ the $Y_k^n$ and $Y^n$ are defined on a common
probability space. Assume that the following diagram holds:
\[
\xymatrixcolsep{5pc}\xymatrix{
Y_k^n \ar[d]_{\mathrm{in\ probability,\ uniformly\ in\ } n }^{k \to\infty} \ar[r]^{(d)}_{n \to\infty}& Y_k \ar[d]_{(d)}^{k\to\infty}\\
Y^n & Y
}
\]

Then $Y^n \mathop{\longrightarrow}\limits^{(d)} Y$.
\end{lemma}

The following is one of our main technical theorems. The proof borrows
heavily from \cite{Janson:GHS}, Theorem~11.16.

%th4.3 #&#
\begin{theorem}\label{sk_theorem}
Let $g \in L^2([0,1]^k \times\R^k)$. Then, as $n \to\infty$,
\[
n^{-{3k}/{4}} \S_k^n(g) \mathop{\longrightarrow}^{(d)} \int
_{[0,1]^k} \int_{\R
^k} g(\vt, \vx)
W^{\otimes k}(d \vt\, d \vx).
\]
Moreover for any finite collection of $k_1, \ldots, k_m \in\N$ and
$g_1, \ldots, g_m$ with $g_i \in L^2([0,1]^{k_i} \times\R^{k_i})$, one
has the joint convergence
\[
\bigl(n^{-{3k_1}/{4}} \S_{k_1}^n(g_1),
\ldots, n^{-{3k_m}/{4}}\S _{k_m}^n(g_m) \bigr)
\mathop{\longrightarrow}^{(d)} \bigl( I_{k_1}(g_1), \ldots,
I_{k_m}(g_m) \bigr).
\]
\end{theorem}

\begin{pf}
The proof proceeds in several steps, beginning with simple functions~$g$ for the $k=1$ case and then bootstrapping to the general case via a
density argument.

\textit{Step} 1. Let $k=1$ and assume that $g(t,x) = \mathbf{1} \{
t_0 < t \leq t_1, x_0 < x \leq x_1 \}$. Then
\[
\S_1^n(g) = 2^{{1}/2} \sum
_{nt_0 < i \leq nt_1} \sum_{\sqrt{n}
x_0 <
x \leq\sqrt{n} x_1} \mathbf{1} \{ i
\leftrightarrow x \} \omega(i,x)
\]
is the sum of order $n^{3/2}(t_1 - t_0)(x_1 - x_0) + O(n)$ mean zero,
variance one random variables [the $2^{{1}/2}$ in front counters the
cancelation of half the $\omega(i,x)$ terms through the $i
\leftrightarrow x$ condition], and hence by the central limit theorem
\[
n^{-{3}/4} \S_1^n(g) \mathop{\longrightarrow}^{(d)} N \bigl(0,
(t_1 - t_0) (x_1 - x_0)
\bigr).
\]
Observing that $\Vert g\Vert _2^2 = (t_1 - t_0)(x_1 - x_0)$, it is easy to see
that\break  $\int_0^1 \int g(t,x) W(dt \,dx)$ has the same distribution. This
completes the theorem in this particular case.

\textit{Step} 2. Still in the $k=1$ case, suppose that $g_1, \ldots,
g_m$ are indicator functions for disjoint, finite area rectangles in
$[0,1] \times\R$. Then the variables $\S_1^n(g_l), l=1,\ldots,m$ are
all independent since they are functions of distinct random variables
$\omega(i,x)$. Moreover their individual limits $I_1(g_l)$ are also
independent, since the distinct $g_l$ are orthogonal in $L^2([0,1]
\times\R)$. Hence from step 1 and standard theory
(see \cite{billingsley}, Chapter~1, Section~4) we have the joint convergence
\[
n^{-{3}/4} \bigl( \S_1^n(g_1), \ldots,
\S_1^n(g_m) \bigr) \mathop{\longrightarrow}^{(d)} \bigl(
I_1(g_1), \ldots, I_1(g_m)
\bigr).
\]

\textit{Step} 3. Recall that a vector of variables converges in law if
and only if all linear combinations of its entries do (also known as
the Cram\'{e}r--Wold device), so that step 2 implies that for all
$\alpha_1, \ldots, \alpha_m \in\R$,
\[
n^{-{3}/4} \sum_{l=1}^{m}
\alpha_l \S_1^n(g_l)
\mathop{\longrightarrow}^{(d)} \sum_{l=1}^m
\alpha_l \int_0^1
g_l(t,x) W(dt \,dx).
\]
However both $\S_1^n$ and the stochastic integral are linear
operators, so
\[
n^{-{3}/4} \S_1^n \Biggl( \sum
_{l=1}^m \alpha_l g_l
\Biggr) \mathop{\longrightarrow}^{(d)} \int_0^1 \sum
_{l=1}^m \alpha_l
g_l(t, x) W(dt\, dx).
\]
Hence the theorem holds for all simple functions.

\textit{Step} 4. In fact joint convergence holds for all finite
collections of simple functions. Let $G_1, \ldots, G_M$ be such a
collection, and observe that for any scalars $\beta_1, \ldots, \beta_M
\in\R$ the sum
\[
\sum_{l=1}^M \beta_l
G_l
\]
is also a simple function, and hence step 3 applies. By linearity this
implies that
\[
\sum_{l=1}^{M} \beta_l
\S_1^n(g_l) \mathop{\longrightarrow}^{(d)} \sum
_{l=1}^M \beta _l \int
_0^1 G_l(t,x) W(dt dx),
\]
and this gives, by the Cram\'{e}r--Wold device, that
\[
n^{-{3}/4} \bigl(\S_1^n(G_1), \ldots,
\S_1^n(G_M) \bigr) \mathop{\longrightarrow}^{(d)}
\bigl(I_1(G_1), \ldots, I_1(G_M)
\bigr).
\]

\textit{Step} 5. We now complete the proof for $k=1$ via a density
argument. For each $g \in L^2([0,1] \times\R)$ there exists a sequence
of simple functions $\{ g_N \}_{N \geq1}$ such that $g_N \to g$ in
$L^2$, as $N \to\infty$. Step 3 gives us that
\[
n^{-{3}/4} \S_1^n(g_N)
\mathop{\longrightarrow}^{(d)} \int_0^1 \int
g_N(t,x) W(dt \,dx)
\]
as $n \to\infty$. By $g_N \to g$ in $L^2$ we have
\[
\int_0^1 \int g_N(t,x) W(dt
\,dx) \mathop{\longrightarrow}^{L^2} \int_0^1
\int g(t,x) W(dt \,dx).
\]
Similarly, by Lemma~\ref{sk_lemma} we have that
\[
Q \bigl[ \bigl( n^{-{3}/4} \bigl(\S_1^n(g_N)
- \S_1^n(g) \bigr) \bigr)^2 \bigr] \leq\Vert
g_N - g\Vert _{L^2}^2.
\]
This implies that, as $N \to\infty$, the left-hand side converges to
zero in $L^2$, uniformly in $n$. These three facts combine with Lemma~\ref{wk_lemma} to give us the following commutative diagram:
\[
\xymatrixcolsep{5pc}\xymatrix{
n^{-3/4} \S_1^n(g_N) \ar[r]^{(d)}_{n \to\infty} \ar[d]_{\mathrm{in}\ L^2,\ \mathrm{uniformly\ in\ } n}^{N \to\infty}& I_1(g_N) \ar[d]_{L^2}^{N \to\infty}\\
n^{-3/4} \S_1^n(g) \ar[r]^{(d)}_{n \to\infty}& I_1(g)}
\]

The joint convergence of any finite collection of $n^{-{3}/4} \S
_1^n(g)$ for $g \in L^2$ is now a consequence of the above, along with
linearity of $\S_1^n$ and the Cram\'{e}r--Wold device, as in step 4.

\textit{Step} 6. Finally, we extend the theorem to $k > 1$. First
consider functions of the form
%
%e30 #&#
\begin{equation}
\label{product_g} g(\vt, \vx) = g_{1}(\vt_1,
\vx_1) \cdots g_{k}(\vt_k, \vx_k),
\end{equation}
with $g_j \in L^2([0,1] \times\R)$ and distinct $g_j$ having distinct
support. For such functions~$g$,
\begin{eqnarray*}
n^{-{k}/{2}} \S_k^n(g) &=& 2^{{k}/{2}}
n^{-{k}/{2}} \sum_{\vi\in E_k^n} \sum
_{\vx\in\Z^k} \prod_{j=1}^k
g_{j} \biggl( \frac
{\vi
_j}{n}, \frac{\vx_j}{\sqrt{n}} \biggr)
\mathbf{1} \{ \vi\leftrightarrow \vx \} \omega(\vi_j,
\vx_j)
\\
&=& \prod_{j=1}^k 2^{{1}/2}
n^{-{1}/2} \sum_{i=1}^n \sum
_{x \in
\Z} g_{j} \biggl( \frac{i}{n},
\frac{x}{\sqrt{n}} \biggr) \mathbf{1} \{ i \leftrightarrow x \} \omega(i,x)
\\
&=& \prod_{j=1}^k n^{-{1}/2}
\S_1^n(g_{j})
\\
&\mathop{\longrightarrow}\limits^{(d)}& \prod_{j=1}^k \int
_0^1 g_{j}(t,x) W(dt \,dx)
\\
&=& \int_{[0,1]^k} \int_{\R^k} g(\vt, \vx) W(d
\vt\, d \vx).
\end{eqnarray*}
The second and last equalities use that the $g_j$ have distinct
supports and \eqref{prod_of_disjoint_g}. This proves the result for
this particular class of functions, and exactly as in step 4 the result
extends to the joint convergence for multiple $g$ of this form,
possibly with different $k$. It only remains to finish the proof for
general $g \in L^2([0,1]^k \times\R^k)$, which is accomplished via a
density argument as in step 5. Functions of the type~\eqref{product_g}
are dense in $L^2([0,1]^k \times\R^k)$, and hence step 5 goes through
with only trivial modifications.
\end{pf}

%re11 #&#
\begin{remark*}
It is worth noting that Theorem~\ref{sk_theorem} does \textit{not}
require $g$ to be symmetric even though the stochastic integrals
$I_k(g)$ are truly only defined for symmetric functions. This is
because the operators $\S_k^n$ have a natural symmetrizing property
already built in. If $\pi\in S_k$, then it is easy to see that
$\overline{(g \circ\pi)}_n(\vt, \vx) = \overline{g}_n(\pi\vt,
\pi\vx
)$, where $g \circ\pi(\vt, \vx) = g(\pi\vt, \pi\vx)$, and therefore
\begin{eqnarray*}
\S_k^n(g \circ\pi) &=& 2^{{k}/{2}} \sum
_{\vi\in E_k^n} \sum_{\vx
\in\Z^k}
\overline{g}_n \biggl( \frac{\pi\vi}{n}, \frac{\pi\vx
}{\sqrt{n}} \biggr)
\mathbf{1} \{ \vi\leftrightarrow\vx \} \omega(\vi, \vx)
\\
&=& 2^{{k}/{2}} \sum_{\vi\in E_k^n} \sum
_{\vx\in\Z^k} \overline {g}_n \biggl( \frac{\vi}{n},
\frac{\vx}{\sqrt{n}} \biggr) \mathbf{1} \bigl\{ \pi ^{-1} \vi
\leftrightarrow\pi^{-1} \vx \bigr\} \omega \bigl(\pi^{-1} \vi,
\pi^{-1} \vx \bigr) = \S_k^n(g),
\end{eqnarray*}
the last equality using that $\mathbf{1} \{ \vi\leftrightarrow\vx \}
\omega
(\vi, \vx)$ is symmetric in its arguments. Hence $\S_k^n(g) = \S
_k^n(\sym g)$.
\end{remark*}

%re12 #&#
\begin{remark*}
If the weight function $g$ is only defined on $\Delta_k \times\R^k$,
then we extend it to $[0,1]^k \times\R^k$ by setting it to be zero for
$\vt\notin\Delta_k$, as in the remark at the end of Section~\ref{wn_si_section}. As $\sym g$ is then just a copy of $g$ on each of the
$k!$ permutations of $[0,1]^k$, we see that the summation over $E_k^n$
may be replaced with a summation over $D_k^n$, with an extra $k!$ for
the proper accounting,
\[
{\S_k^n}^*(g):= 2^{{k}/{2}} \sum
_{\vi\in D_k^n} \sum_{\vx\in
\Z
_k^n}
\overline{g}_n \biggl( \frac{\vi}{n}, \frac{\vx}{\sqrt{n}} \biggr)
\mathbf{1} \{ \vi\leftrightarrow\vx \} \omega(\vi, \vx) = k! \S_k^n(g).
\]
Again by the remark of Section~\ref{wn_si_section} this means that
\[
n^{-{3k}/{4}} {\S_k^n}^*(g) \mathop{\longrightarrow}^{(d)} \int
_{\Delta_k} \int_{\R^k} g(\vt, \vx) W(d \vt\, d
\vx).
\]
In the future we will drop the $*$ notation on $\S_k^n$ as it is
usually understood from the context what the domain of $g$ is.
\end{remark*}

Finally we state a theorem which summarizes exactly how discrete Wiener
chaos expansions converge to continuum ones in the $L^2$ sense.

%le4.4 #&#
\begin{lemma}\label{discrete_to_cont}
If $g = (g_0, g_1, g_2, \ldots) \in\bigoplus_{k \geq0} L^2([0,1]^k
\times\R^k)$, then as \mbox{$n \to\infty$}
\[
I^n(g):= \sum_{k=0}^{\infty}
n^{-{3k}/{4}} \S_k^n(g_k) \mathop{\longrightarrow}^{(d)} \sum_{k=0}^{\infty} \int
_{[0,1]^k} \int_{\R^k} g_k(\vt,
\vx) W^{\otimes k}( d\vt\, d \vx) = I(g).
\]
Moreover if $G_1, \ldots, G_m \in\bigoplus_{k \geq0} L^2([0,1]^k
\times\R^k)$, then as $n \to\infty$ we have the joint convergence
\[
\bigl( I^n(G_1), \ldots ,I^n(G_m)
\bigr) \mathop{\longrightarrow}^{(d)} \bigl(I(G_1), \ldots, I(G_m)
\bigr).
\]
\end{lemma}

\begin{pf}
Since $I$ is an isometry from the Fock space onto $L^2(\Omega_W, \F_W,
\envW)$, we automatically get that $\sum_{k = 0}^M I_k(g_k)
\rightarrow
\sum_{k=0}^{\infty} I_k(g_k)$ in $L^2(\Omega_W, \F_W, \envW)$, as $M
\to\infty$. Since $\Var( n^{-{3k}/{4}} \S_k^n(g_k) ) \leq\Var(
I_k(g_k) )$, this also implies that
\[
\sum_{k=0}^M n^{-{3k}/{4}}
\S_k^n(g_k) \mathop{\longrightarrow}^{M \to \infty} \sum
_{k=0}^{\infty} n^{-{3k}/{4}}
\S_k^n(g_k)
\]
in $L^2(\Omega, Q)$, uniformly in $n$. Theorem~\ref{sk_theorem}
implies that
\[
\sum_{k=0}^M n^{-{3k}/{4}}
\S_k^n(g_k) \mathop{\longrightarrow}^{(d)} \sum
_{k=0}^M I_k(g_k)
\]
as $n \to\infty$. Putting these pieces together and using Lemma~\ref
{wk_lemma} gives us the diagram
\[
\xymatrixcolsep{5pc}\xymatrix{
\sum_{k=0}^M n^{-3k/4}\S_k^n(g_k) \ar[d]_{\mathrm{uniformly\ in\ } L^2}^{M \to\infty }\ar[r]^{(d)}_{n \to\infty}& \sum_{k=0}^M I_k(g_k)\ar[d]_{L^2}^{M \to\infty}\\
% \\
%@V{V @VV{}V \
\sum_{k=0}^{\infty} n^{-3k/4} \S_k^n(g_k) \ar[r]^{(d)}_{n \to\infty}&
\sum_{k=0}^{\infty} I_k(g_k).
}
\]
The joint convergence follows by another application of the Cram\'
{e}r--Wold device.
\end{pf}

%s4.2 #&#
\subsection{Perturbations of the environment field}

Our strategy for proving Theorem~\ref{Zn_convergence_theorem} will be
to first prove the convergence result for $\symz_n^{\omega}$, and
then extend it
to~$Z_n^{\omega}$. The main idea is that, after a proper deterministic
normalization, $Z_n^{\omega}$ can be written in the same form as
$\symz_n^{\omega}
$, but with a mean and variance that may only be \textit
{asymptotically} zero and one, respectively. In this section we give
sufficient conditions on the mean and variance so that Theorem~\ref
{sk_theorem} still holds.

Throughout we let $\tilde{\omega}_n(i,x)$ denote a field of i.i.d.
random variables on $\N\times\Z$. The dependence on $n$ is to
indicate that the distribution of the environment variables may vary
with $n$. For $g \in L^2([0,1]^k \times\R^k)$ we let $\S_k^n(g;
\tilde
{\omega}_n)$ denote the same quantity as $\S_k^n(g)$, but with the
$\omega$ variables replaced by $\tilde{\omega}_n$. We have the
following generalization of Theorem~\ref{sk_theorem}:

%th4.5 #&#
\begin{theorem}\label{wn_ustat_extension}
Assume that the environment variables $\tilde{\omega}_n$ satisfy\break
$Q(\tilde{\omega}_n) = 0$ and $Q(\tilde{\omega}_n^2) = 1 + o(1)$. Then
the statement of Theorem~\ref{sk_theorem} holds with all instances of
$\S_k^n(g)$ replaced by $\S_k^n(g; \tilde{\omega}_n)$.
\end{theorem}

\begin{pf}
Re-examining the proof of Theorem~\ref{sk_theorem}, we see that it is
enough to check that steps $1$ and $5$ are still valid; the other steps
are unchanged. For step $1$ we again assume that $g(t,x) = \mathbf{1}
\{ t_0 < t \leq t_1, x_0 < x \leq x_1 \}$, and we want to show that
$n^{-3/4} \S_1^n(g; \tilde{\omega}_n)$ converges in law to a Gaussian
with mean zero and variance $\Vert g\Vert _2^2$. Since the distribution of the
random variables is allowed to change with $n$ we require extra
hypotheses found in the central limit theorems for triangular arrays.
By \cite{durrett:book}, Theorem~2.4.5, the given hypothesis on the mean
and variance are sufficient. The cited theorem also requires a
condition of the form $Q(\tilde{\omega}_n^2 \mathbf{1} \{ |\tilde
{\omega }_n| > \varepsilon\sqrt{n} \}) \to0$ for every $\varepsilon> 0$,
but since
the $\tilde{\omega}_n$ variables are i.i.d. (for each $n$) the
hypothesis on the variance covers this; see the remark after
\cite{durrett:book}, Theorem~2.4.5.

For step $5$ observe that we now have
\[
Q \Biggl[ \prod_{j=1}^k \tilde{
\omega}_n(\vi_j, \vx_j) \prod
_{j=1}^k \tilde {\omega}_n \bigl(
\vi'_j, \vx'_j \bigr)
\Biggr] = \bigl(1 + o(1) \bigr)^k \mathbf{1} \bigl\{ \vi= \vi
', \vx= \vx' \bigr\}.
\]
Correspondingly, we have as an analogue of Lemma~\ref{sk_lemma}
\[
Q \bigl( \S_k^n(g; \tilde{\omega}_n)^2
\bigr) \leq n^{{3k}/{2}} \bigl(1+o(1) \bigr)^k Q \bigl(\tilde{
\omega}_n^2 \bigr)^k \Vert g\Vert
^2_{L^2([0,1]^k \times\R^k)}.
\]
Now for $g \in L^2([0,1] \times\R)$ again let $g_N$ be a sequence of
simple functions such that $g_N \rightarrow g$ in $L^2$. All that needs
to be checked is that $n^{-{3}/4} \S_1^n(g_N; \tilde{\omega}_n)
\to
n^{-{3}/4} \S_1^n(g; \tilde{\omega}_n)$ in $L^2$, uniformly in $n$,
but this is obvious since by the last calculation,
\[
Q \bigl( \bigl( n^{-{3}/4} \S_1^n(g_N
- g; \tilde{\omega}_n) \bigr)^2 \bigr) \leq \bigl(1 +
o(1) \bigr) \Vert g_N - g\Vert ^2_{L^2}.
\]
\upqed\end{pf}

We now extend Lemma~\ref{discrete_to_cont} to the situation of the
previous lemma.

%le4.6 #&#
\begin{lemma}\label{discrete_to_cont_extended}
Let $\tilde{\omega}_n$ satisfy the hypothesis of the previous theorem,
and assume that $g = (g_0, g_1, \ldots, ) \in\bigoplus_{k \geq0}
L^2([0,1]^k \times\R^k)$ is such that
%
%e31 #&#
\begin{equation}
\label{tail_condition} \lim_{N \to\infty} \limsup_{n \to\infty}
\sum_{k=N}^{\infty} Q \bigl(\tilde {
\omega}_n^2 \bigr)^k \Vert g_k
\Vert ^2_{L^2} = 0,
\end{equation}
then
\[
\tilde{I}^n(g):= \sum_{k=0}^{\infty}
n^{-3k/4} \S_k^n(g_k; \tilde {
\omega}_n) \mathop{\longrightarrow}^{(d)} I(g)
\]
as $n \to\infty$. Moreover if $G_1, \ldots, G_M \in\bigoplus
L^2([0,1]^k \times\R^k)$ individually satisfy~\eqref
{tail_condition}, then
\[
\bigl( \tilde{I}^n(G_1), \ldots, \tilde{I}^n(G_M)
\bigr) \mathop{\longrightarrow}^{(d)} \bigl(I(G_1), \ldots, I(G_M)
\bigr).
\]
\end{lemma}

\begin{pf}
The proof of Lemma~\ref{discrete_to_cont} goes through as before; one
only needs to check that
\[
\sum_{k=0}^{M} n^{-3k/4}
\S_k^n(g_k; \tilde{\omega}_n)
\rightarrow \sum_{k=0}^{\infty} n^{-3k/4}
\S_k^n(g_k; \tilde{\omega}_n)
\]
uniformly in $n$ as $M \to\infty$, and \eqref{tail_condition}
guarantees that this is true in the $L^2$ sense. The joint convergence
again follows from the Cram\'{e}r--Wold device: one only needs to check
that condition \eqref{tail_condition} is satisfied by any linear
combination of the $G_i$, and this is an easy consequence of the
triangle inequality.~%
\end{pf}

%re13 #&#
\begin{remark*}
The assumption that $Q(\tilde{\omega}_n) = 0$ is not really necessary;
in general $Q(\tilde{\omega}_n) = o(n^{-3/4})$ will suffice. However,
in this case the cross product terms take the form
\begin{eqnarray*}
&&Q \Biggl[ \prod_{j=1}^k \tilde{
\omega}_n(\vi_j, \vx_j) \prod
_{j=1}^k \tilde {\omega}_n \bigl(
\vi'_j, \vx'_j \bigr)
\Biggr] \\
&&\qquad= Q \bigl( \tilde{\omega}_n^2
\bigr)^{\# \{ j \dvtx (\vi
_j, \vx_j) = (\vi'_j, \vx'_j) \}} Q(\tilde{\omega}_n)^{2 \# \{ j \dvtx (\vi
_j, \vx_j) \neq(\vi'_j, \vx'_j) \}},
\end{eqnarray*}
which is cumbersome to deal with.
\end{remark*}

%s5 #&#
\section{Convergence of the point-to-line partition functions}
\label{part_function_convergence_section}

In this section we use the results of Section~\ref{U_stat_section} to
prove Theorem~\ref{Zn_convergence_theorem} on convergence of the
partition functions. The strategy is to prove the convergence for the
modified partition functions $\symz_n^{\omega}$ and then transfer the
results to
the usual partition functions $Z_n^{\omega}$. Throughout we assume that
the environment variables have mean zero and variance one.

%de5.1 #&#
\begin{definition}
For $k, n \geq1$, define $p_k^n \dvtx [0,1]^k \times\R^k \to\R$ by
\[
p_k^n(\vt, \vx) = \overline{p}_k \bigl(
\lceil n \vt \rceil, \vx\sqrt{n} \bigr) \mathbf{1} \bigl\{ \lceil n \vt \rceil \in
D_k^n \bigr\}.
\]
\end{definition}

%re14 #&#
\begin{remark*}
Observe that for $k > n$, it is impossible for a vector in $[0,1]^k$ to
have all $k$ elements separated by at least $1/n$. Hence the condition
$\lceil n \vt \rceil \in D_k^n$ implies that $p_k^n$ is identically zero.
\end{remark*}

%re15 #&#
\begin{remark*}
Temporarily ignoring the indicator function term in this last
definition, $p_k^n$ is the $k$-fold density function of the continuous
time process
\[
t \mapsto\frac{\S_{\lceil nt \rceil} + U_{\lceil nt \rceil}}{\sqrt{n}},
\]
where the $U_i$ are an i.i.d. collection of Uniform random variables on
$(-1,1)$. We will use this later to simplify definitions for the
point-to-point partition functions.\vspace*{-2pt}
\end{remark*}

%re16 #&#
\begin{remark*}
Also observe that $p_k^n$ is already constant on the rectangles of
$\mathcal{R}_k^n$, so that $\overline{p_k^n} = p_k^n$. Moreover, for
$\vi\in\E_k^n, \vx\in\Z^k$ such that $\vi\leftrightarrow\vx$,
\[
p_k^n \biggl( \frac{\vi}{n}, \frac{\vx}{\sqrt{n}}
\biggr) = \overline {p}_k(\vi, \vx ) \mathbf{1} \bigl\{ \vi\in
D_k^n \bigr\} = 2^{-k} p_k(\vi,
\vx) \mathbf{1} \bigl\{ \vi\in D_k^n \bigr\}.
\]
Thus, by definition \eqref{sk_def} of $\S_k^n$,
\[
\S_k^n \bigl(p_k^n \bigr) =
2^{-{k}/{2}} \sum_{\vi\in D_k^n} \sum
_{\vx\in
\Z
^k} p_k(\vi, \vx) \omega(\vi, \vx).
\]
The $\vi\leftrightarrow\vx$ condition is already handled by $p_k$.
This leads to the following:\vspace*{-2pt}
\end{remark*}

%le5.2 #&#
\begin{lemma}\label{p2l_rewrite}
The point-to-line partition function may be rewritten as
\[
\symz_n^{\omega}(\beta) = \sum_{k=0}^{n}
2^{{k}/{2}} \beta^k \S _k^n
\bigl(p_k^n \bigr) = \sum_{k=0}^n
2^{{k}/{2}} \beta^k n^{-{k}/{2}} \S_k^n
\bigl( n^{{k}/{2}}p_k^n \bigr).
\]
Consequently,
\[
\symz_n^{\omega} \bigl(\beta n^{-{1}/4} \bigr) = \sum
_{k=0}^n 2^{{k}/{2}}
\beta^k n^{-{3k}/{4}} \S_k^n \bigl(
n^{{k}/{2}} p_k^n \bigr).\vspace*{-2pt}
\]
\end{lemma}

From this we have the following:

%pr5.3 #&#
\begin{proposition}\label{zbar_convergence_prop}
Assume that the environment variables $\omega$ have mean zero and
variance one. Then as $n \to\infty$ we have $\symz_n^{\omega}(\beta
n^{-{1}/{4}})
\mathop{\longrightarrow}\limits^{(d)} \mathcal{Z}_{\sqrt{2}
\beta}$.\vspace*{-2pt}
\end{proposition}

\begin{pf}
First observe that Lemma~\ref{discrete_to_cont} implies that for all
$\beta> 0$,
\[
\sum_{k=0}^{\infty} \beta^k
n^{-3k/4} \S_k^n(\varrho_k)
\mathop{\longrightarrow}^{(d)} \mathcal{Z}_{\beta}
\]
as $n \to\infty$. Now we show that the difference between this term
and $\symz_n^{\omega}(\beta n^{-{1}/{4}})$ goes to zero as $n
\to\infty$. Observe that
\begin{eqnarray*}
&&\sum_{k=0}^{\infty}  2^{k/2}
\beta^k n^{-3k/4} \S_k^n(
\varrho_k) - \symz_n^{\omega} \bigl(\beta
n^{-{1}/{4}} \bigr)
\\
&&\qquad= \sum_{k=0}^n 2^{k/2}
\beta^k n^{-3k/4} \S_k^n \bigl(
\varrho_k - n^{
{k}/{2}} p_k^n \bigr) +
\sum_{k=n+1} 2^{k/2} \beta^k
n^{-3k/4} \S _k^n(\varrho_k).
\end{eqnarray*}
By Lemma~\ref{sk_lemma} the second term is bounded in $L^2$ by
\[
\sum_{k=n+1}^{\infty} 2^k
\beta^{2k} \Vert \varrho_k\Vert ^2_{L^2([0,1]^k
\times\R^k)},
\]
which, by the estimates in Section~\ref{brownian_chaos_section}, goes
to zero as $n \to\infty$. For the first term, note that in $L^2$ it is
bounded above by
\[
\sum_{k=0}^n 2^{k}
\beta^{2k} \bigl\Vert \varrho_k - n^{{k}/{2}}
p_k^n\bigr\Vert _{L^2}^2.
\]
The local limit theorem implies that $n^{k/2} p_k^n \to\varrho_k$
pointwise as $n \to\infty$. In Lemma~\ref{lemma:discrete_bound} of the
\hyperref[app]{Appendix} we prove there is a constant $C > 0$ such that
\[
\sup_n \bigl\Vert n^{k/2} p_k^n
\bigr\Vert _{L^2} \leq C^k \Vert \varrho_k\Vert
_{L^2},
\]
hence by the triangle inequality and dominated convergence we have
$\Vert \varrho_k - n^{{k}/{2}} p_k^n\Vert _{L^2} \to0$ as $n \to\infty$.
Since the sequence $C^k \Vert \varrho_k\Vert _{L^2}$ is summable, the estimate
above and dominated convergence also imply that
\[
\lim_{n \to\infty} \sum_{k=0}^n
2^k \beta^{2k} \bigl\Vert \varrho_k -
n^{
{k}/{2}} p_k^n\bigr\Vert _{L^2}^2
= \sum_{k=0}^{\infty} 2^k
\beta^{2k} \lim_{n
\to\infty} \bigl\Vert \varrho_k -
n^{{k}/{2}} p_k^n\bigr\Vert _{L^2}^2
= 0.
\]
\upqed\end{pf}

Now we begin the process of extending the convergence to the partition
function~$Z_n^{\omega}$.

%pr5.4 #&#
\begin{proposition}\label{prop:Z_p2l_convergence}
Suppose that there is a $\beta_0 > 0$ such that $\lambda(\beta):=
\log
Q ( e^{\beta\omega} ) < \infty$ for all $0 < \beta< \beta_0$. Then
\[
e^{-n \lambda(\beta n^{-1/4})} Z_n^{\omega} \bigl(\beta n^{-{1}/{4}}
\bigr) \mathop{\longrightarrow}^{(d)} \mathcal{Z}_{\sqrt{2} \beta}.
\]
\end{proposition}

\begin{pf}
Define the environment field $\tilde{\omega}_n$ by
\[
e^{\beta n^{-1/4} \omega(i,x) - \lambda(\beta n^{-1/4})} = 1 + \beta
n^{-{1}/4} \tilde{\omega}_n(i,x),
\]
so that
\begin{eqnarray*}
 Z_n^{\omega} \bigl(\beta n^{-{1}/{4}} \bigr)& = &\Pl \Biggl[
\prod_{i=1}^n \bigl(1 + \beta
n^{-{1}/{4}}\tilde{\omega }_n(i,S_i) \bigr) \Biggr]
\\
&= &\symz_n^{\tilde{\omega}_n} \bigl(\beta n^{-{1}/{4}} \bigr).
\end{eqnarray*}
It is straightforward to check that exponential moments for $\omega$
and the definition of $\lambda(\beta)$ imply that $\tilde{\omega}_n$
satisfy $Q(\tilde{\omega}_n) = 0$ and $Q(\tilde{\omega}_n^2) = 1 +
O(n^{-1/4})$, which are even stronger than the hypotheses of Theorem~\ref{wn_ustat_extension}. The proof is now completed by using Lemma~\ref
{discrete_to_cont_extended} and mimicking the proof of Proposition~\ref
{zbar_convergence_prop}; briefly, it follows that
\[
\sum_{k=0}^{\infty} \beta^k
n^{-3k/4} \S_k^n(\varrho_k; \tilde {
\omega }_n) \mathop{\longrightarrow}^{(d)} \mathcal{Z}_{\beta}
\]
as $n \to\infty$. This is clear by Lemma~\ref
{discrete_to_cont_extended} since $C^k \Vert  \varrho_k \Vert _{L^2}$ is
summable in $k$ for any $C > 0$. Furthermore
\[
\sum_{k=0}^{\infty} 2^{k/2}
\beta^k \S_k^n \bigl(\varrho_k -
n^{{k}/{2}} p_k^n; \tilde{\omega}_n
\bigr)
\]
also goes to zero for the same reason as in Proposition~\ref
{zbar_convergence_prop}, by using Lemma~\ref{lemma:discrete_bound} as before.
\end{pf}

%s6 #&#
\section{Convergence of the point-to-point partition functions}

In this section we describe the proofs behind Theorems \ref
{p2p_zn_convergence_theorem} and \ref{theorem:4_param_convergence}. The
arguments are based on the ones for the point-to-line partition function.

%s6.1 #&#
\subsection{The random local limit theorem}
\label{sec:rllt}

We extend the methods of the last section to prove convergence of the
endpoint density under intermediate disorder. We will prove all parts
of Theorem~\ref{p2p_zn_convergence_theorem} except for the tightness,
the proof of which is delayed until the \hyperref[app]{Appendix}. Much of what we
describe in this section is a relatively simple extension of the
convergence of the point-to-line partition function so we do not go
into as much detail.

Appealing to equation \eqref{p2p_expansion}, we easily see that the term
\[
\symz_n^{\omega} \bigl(x \sqrt{n}; \beta n^{-{1}/{4}}
\bigr)
\]
may be written as the sum of $n$ terms in a discrete Wiener chaos
expansion, in an analogous way to what was done in Lemma~\ref
{p2l_rewrite}. Define $p_{k|x}^n$ on $[0,1]^k \times\R^k$ by
\[
n^{k/2} p_{k|x}^n(\vt, \vx) = \Pl(
X_{\vt_1} \in d\vx_1, \ldots, X_{\vt
_k} \in d
\vx_k | X_1 \in dx ) \mathbf{1} \bigl\{ \lceil n \vt
\rceil \in D_k^n \bigr\},
\]
where $X_t = (S_{\lceil nt \rceil} + U_{\lceil nt \rceil})/\sqrt{n}$
is the
continuous time process as defined before in Section~\ref{sec2.1}. Using this,
we may rewrite \eqref{p2p_expansion} as
%
%e32 #&#
\begin{eqnarray}
\label{p2p_as_sk_sum} &&\Pl \Biggl[  \prod_{i=1}^n
\bigl( 1 + \beta n^{-{1}/{4}}\omega(i, S_i) \bigr) \Big|
S_n = x \sqrt{n} \Biggr]
\nonumber
\\[-8pt]
\\[-8pt]
\nonumber
&&\qquad = \sum_{k=0}^n
2^{{k}/{2}} \beta^k n^{-{3k}/{4}} \S_k^n
\bigl( n^{
{k}/{2}} p_{k|x}^n \bigr).
\end{eqnarray}
On the right-hand side the $\sqrt{n}$ spatial scaling on $x$ has
already been factored into the definition of $p_{k|x}^n$. On the other
hand, defining $\varrho_{k|x}$ to be the $k$-fold density of a Brownian
bridge from $0$ to $x$, that is,
\[
\varrho_{k|x}(\vt, \vx) = \frac{\varrho(1-\vt_k, x - \vx
_k)}{\varrho(1,
x)} \varrho_k(
\vt, \vx),
\]
it follows immediately from Lemma~\ref{discrete_to_cont} that
\[
\Biggl\{ x \mapsto\sum_{k=0}^{\infty}
n^{-3k/4} \S_k^n(\varrho_{k|x}) \Biggr\}
\mathop{\longrightarrow}^{(d)} \bigl\{ x \mapsto e^{A_{\beta}(x)} \bigr\}
\]
as $n \to\infty$, in the sense of convergence of finite dimensional
distributions; see~\eqref{Zbeta_x_def} for the definition of $A_{\beta
}(x)$. All that remains to be shown is that the difference between this
process and
\[
x \mapsto\sum_{k=0}^n n^{-3k/4}
\S_k^n \bigl(n^{k/2} p_{k|x}^n
\bigr)
\]
goes to zero as $n \to\infty$. The difference breaks into two terms,
the easiest of which to deal with is
\[
\sum_{k=n+1}^{\infty} 2^{k/2}
\beta^k n^{-3k/4} \S_k^n(
\varrho_{k|x}).
\]
By Lemma~\ref{sk_lemma} and the comments at the end of Section~\ref{brownian_chaos_section} this term goes to zero in $L^2$ at a rate that
is uniform in $x$. The other term to deal with is
\[
\sum_{k=0}^n 2^{k/2}
\beta^k n^{-{3k}/{4}} \S_k^n \bigl(\varrho
_{k|x} - n^{k/2} p_{k|x}^n \bigr),
\]
which has variance bounded above by
\[
\sum_{k=0}^n 2^k
\beta^{2k} \bigl\Vert \varrho_{k|x} - n^{k/2}
p_{k|x}^n \bigr\Vert _{L^2}^2.
\]
To show this latter term goes to zero requires another application of
dominated convergence as in the last section, and this uses the
estimates of Lemma~\ref{lemma:discrete_bound}.

To move from convergence of \eqref{p2p_as_sk_sum} to convergence of
$\symz_n^{\omega}(x \sqrt{n}; \beta n^{-{1}/{4}})$ is a simple
matter since the two differ by
only a factor of $p(n, x \sqrt{n})$. The local limit theorem
immediately implies that $\sqrt{n} p(n, x\sqrt{n})/2$ converges to
$\rho
(1,x)$ as $n \to\infty$, and modulo the tightness this completes the
proof of \eqref{Zbeta_on_parabola} that
\[
\biggl\{ x \mapsto\frac{\sqrt{n}}{2} \symz_n^{\omega} \bigl(x
\sqrt{n}; \beta n^{-{1}/{4}} \bigr) \biggr\} \mathop{\longrightarrow}^{(d)} \bigl\{ x
\mapsto e^{A_{\sqrt{2} \beta}(x)} \varrho(1,x) \bigr\}.
\]
Since Lemma~\ref{discrete_to_cont} also implies joint convergence of
the point-to-line partition function \textit{and} the point-to-point
partition function at any finite collection of points, this implies the
random local limit theorem for the endpoint density
\[
\biggl\{ x \mapsto\frac{\sqrt{n} \symz_n^{\omega}(x \sqrt{n}; \beta
n^{-{1}/{4}})/2}{ \symz_n^{\omega}(\beta n^{-{1}/{4}}
) } \biggr\} \mathop{\longrightarrow}^{(d)} \biggl\{ x
\mapsto\frac{e^{A_{\beta}(x)} \varrho
(1,x)}{\mathcal{Z}_{\beta}} \biggr\}.
\]

Finally it only remains to extend the result for the $\symz_n^{\omega
}$ partition
functions to the $Z_n^{\omega}$ ones. This is accomplished the same way
as in the last section by introducing the $\tilde{\omega}_n$
environment field and using the relation $Z_n^{\omega} = \symz
_n^{\tilde
{\omega}_n}$. The only extra work required is in showing that the
conditions of Lemma~\ref{discrete_to_cont_extended} are satisfied, but,
as for the point-to-line partition function, this is an easy
consequence of the estimates in the \hyperref[app]{Appendix}.

%s6.2 #&#
\subsection{Convergence of the four-parameter field}
\label{sec:4_param_field}

Now we prove convergence of the four-parameter field of transition
probabilities as stated in Theorem~\ref{theorem:4_param_convergence}.
As in the last section we defer the tightness until the \hyperref[app]{Appendix} and
concentrate on convergence of the finite dimensional distributions.
This follows the same scheme as before. Partition functions of the form
\[
\symz^{\omega}(m, y; k, x; \beta)
\]
[see \eqref{defn:4_param_modified} for the definition] are
point-to-point partition functions shifted to a different starting
point. Using the techniques of the last section we know that its law
converges under intermediate disorder scaling for the environment if
space and time are scaled diffusively, and hence this implies that
\[
\frac{\sqrt{n}}{2} \symz^{\omega} \bigl(ns, y \sqrt{n}; nt, x \sqrt{n};
\beta n^{-1/4} \bigr) \mathop{\longrightarrow}^{(d)} \mathcal{Z}_{\sqrt{2} \beta}(s,y;t,x)
\]
as $n \to\infty$. This can also be seen by expanding \eqref
{defn:4_param_modified} as a discrete chaos series of type \eqref
{p2p_as_sk_sum}, using kernels of the form
\begin{eqnarray*}
&& n^{{k}/{2}} p_{k|(s,y;t,x)}^n (\vt, \vx)\\
&&\qquad:= \Pl(
X_{\vt_1} \in d\vx _1, \ldots, X_{\vt_k} \in d
\vx_k | X_s \in dy, X_t \in dx ) \mathbf{1}
\bigl\{ \lceil nt \rceil \in D_k^n \bigr\}.
\end{eqnarray*}
These kernels are space--time shifts of the kernels $n^{k/2} p_{k|x}^n$.
Shift invariance of the underlying random walk implies that these
kernels are simple translates of each other. Since shifting the kernels
is equivalent to shifting the field of environment random variables,
and the law of the field is clearly shift invariant, the shifted
partition function is equal in law to the unshifted one. More
precisely, we have equality in law of
\[
\frac{\sqrt{n}}{{2}} \symz^{\omega} \bigl(ns, y \sqrt{n}; nt, x \sqrt {n};
\beta n^{-1/4} \bigr) \equiv\frac{\sqrt{n}}{2} \symz^{\omega}
\bigl(0,0; n(t-s), (x-y) \sqrt{n}; \beta n^{-1/4} \bigr).
\]
The law of the right-hand side converges by the arguments in Section~\ref{sec:rllt}, which is just the special case of $s = 0$ and $t = 1$.
For a finite collection of space--time points $(s_i, y_i; t_i, x_i)$ the
joint convergence of
\[
\frac{\sqrt{n}}{2} \symz^{\omega} \bigl(n s_i,
y_i \sqrt{n}; n t_i, x_i \sqrt{n};
\beta n^{-1/4} \bigr)
\]
follows by Lemma~\ref{discrete_to_cont}.

Similarly, partition functions of the form
\[
\symz_n^{\omega}(m, y; k, *; \beta)
\]
are point-to-line partition functions shifted to a different starting
point. They can be expanded into discrete Wiener chaos using the kernel
functions
\[
n^{{k}/{2}} p_{k|(s,y;t,*)}^{n}(\vt, \vx) = \Pl(
X_{\vt_1} \in d \vx_1, \ldots, X_{\vt_k} \in d
\vx_k | X_s \in dy ).
\]
Using the methods of Section~\ref{part_function_convergence_section}
their law converges under the intermediate disorder scaling on the
environment and diffusive scaling on space and time. Joint convergence of
\[
\symz^{\omega} \bigl(ns_i, y_i \sqrt{n};
nt_i, *; \beta n^{-1/4} \bigr)
\]
at finitely many points $(s_i, y_i; t_i)$ follows from Lemma~\ref
{discrete_to_cont}.

To convert convergence of the $\symz^{\omega}$ partition functions
into the
$Z^{\omega}$ partition functions, introduce again the field $\tilde
{\omega}_n$ and use the relation $Z^{\omega} = \symz^{\tilde{\omega
}_n}$. For example, it is clear under this definition that
\begin{eqnarray*}
&&\symz^{\tilde{\omega}_n} \bigl(ns, y \sqrt{n}; nt, x \sqrt{n}; \beta
n^{-1/4} \bigr)\\
&&\qquad = e^{-n(t-s) \lambda(\beta n^{-1/4})} Z^{\omega} \bigl(ns, y \sqrt
{n}; nt, x \sqrt{n}; \beta n^{-1/4} \bigr),
\end{eqnarray*}
up to a negligible difference caused by the term $\exp\{ -n(t-s)
\lambda(\beta n^{-1/4}) \}$ (since $t - s$ is not usually a multiple of
$1/n$). The left-hand side converges from the arguments just discussed,
and joint convergence at a finite collection of space--time points
$(s_i, y_i; t_i, x_i)$ follows from Lemma~\ref
{discrete_to_cont_extended}. A similar argument shows convergence of
the finite dimensional distributions for the field of $Z^{\omega}$
point-to-line partition functions.

%s7 #&#
\section{Concluding remarks}
\label{sec:conclusion}

We end with a few brief observations and ideas for future work.

%s7.1 #&#
\subsection{Supercritical scaling}

The $n^{-1/4}$ scaling on the environment is the hallmark of the
intermediate disorder regime. Under this scaling the law of the random
polymer measure converges, although with probability one the polymer
measure itself does not converge. In this respect the $n^{-1/4}$
scaling is sharp, meaning that if one replaces $1/4$ by any larger
exponent then the resulting polymer would be in the weak disorder
regime. More precisely:

\begin{mainn*}
Under the scaling $\beta_n = \beta n^{-({1}/4 + \delta)}$ for any \mbox{$\delta> 0$}:
\begin{itemize}
\item the partition function $e^{-n \lambda(\beta_n)} Z_n^{\omega
}(\beta n^{-(1/4 + \delta)})$ converges in probability to~$1$;
\item the endpoint density, under diffusive scaling of space, converges
to the standard Gaussian distribution;
\item the transition probabilities converge, under diffusive scaling of
space and time, to the transition probabilities for standard Brownian motion.
\end{itemize}
\end{mainn*}

The idea behind the proof of these results is already apparent in the
proofs of Theorems \ref{Zn_convergence_theorem} and \ref
{p2p_zn_convergence_theorem}. The extra term of $n^{-\delta}$ in the
exponent sends each of the random terms of \eqref{Zbar_expansion2} and
\eqref{p2p_expansion} to zero (in $L^2$), and all that survives is the
deterministic first-order terms. These correspond to Gaussian endpoint
fluctuations. Full details for the point-to-line partition function are
given in~\cite{feng:diffusive}.

%s7.2 #&#
\subsection{Subcritical scaling}

Replacing the $1/4$ exponent by a smaller value produces a regime that
we are unable to analyze rigorously but for which we have many
conjectures. Consider the scalings $\beta_n:= \beta n^{-\alpha}$ for
$0 \leq\alpha< 1/4$. Under this scaling our methodology breaks down
because the individual terms of the discrete Wiener chaos blow up as $n
\to\infty$. Our conjectured value for the fluctuation exponents are
\[
\chi(\alpha) = \tfrac{1}3 (1 - 4\alpha), \qquad\zeta(\alpha) =
\tfrac{2}3(1-\alpha).
\]
Observe that these values linearly interpolate between the conjectured
values of $\zeta= 2/3, \chi=1/3$ at $\alpha= 0$, and the values
$\zeta= 1/2, \chi=0$ that we prove in this paper. Moreover, $\chi
(\alpha)$ and $\zeta(\alpha)$ satisfy the relation
\[
\chi(\alpha) = 2 \zeta(\alpha) - 1,
\]
which is already predicted in the $\alpha= 0$ case. Further details of
this conjecture (and others) are discussed in \cite{AKQ:prl}.

%s7.3 #&#
\subsection{Assumptions on moments}

In this paper we have two different sets of assumptions on the moments
of the environment random variables, each set corresponding to a
different Hamiltonian used to construct the polymer measure. In the
first case, for the partition function $\symz_n^{\omega}$, all of our
results on
intermediate disorder only require that the environment variables have
finite variance for convergence of the finite dimensional
distributions. This is essentially a consequence of the central limit
theorem. However, for the more commonly studied partition function
$Z_n^{\omega}$, it is clear that more moments are required. In order
for the partition function $Z_n^{\omega}(\beta)$ to have finite mean it
is necessary that the environment variables have finite exponential
moments. It is possible, however, that even without this assumption the
$Z_n^{\omega}$ partition function converges to the same limits as
before, and we believe this to be the case. In fact we believe that it
is sufficient that the environment variables have more than $6$
moments. Loosely speaking, this conjecture is based on the idea that
the path measure with enough moments has diffusive scalings under
intermediate disorder. This suggests that $n^{3/2}$ environment
variables are all that is contributing to the partition function, and
as long as no single one of these variables is dominant, the
convergence will be enforced by the central limit theorem. Since each
environment variable is multiplied by a factor of $n^{-1/4}$ under
intermediate disorder, a simple Chebyshev bound shows that greater than
$6$ moments is sufficient to keep all of them order one. A similar
argument indicates that the condition of $6$ moments is sharp, and that
partition functions with less than $6$ moments do not converge to the
universal limits in the intermediate disorder regime. It is interesting
to compare this moment condition with the case of the unscaled
environment (strong disorder), where it is widely believed \cite
{BBP:top_eigenvalue} that greater than $4$ moments is sufficient for
convergence to the universal limit governed by Tracy--Widom laws.

%s7.4 #&#
\subsection{Crossover on the process level}

As was explained, under proper scalings the one-point marginal
distributions of the process $A_{\beta}(x)$ cross over from a Gaussian
distribution to the Tracy--Widom GUE. It is natural then to conjecture
that the process converges from a stationary Gaussian process to the
$\Ai_2$ process, with proper scalings in the $x$ variable as $\beta$
approaches $0$ and $\infty$. It is tempting to try to prove at least
the tightness in the limit as $\beta\to\infty$ using the control
provided by the convergent power series for the process $A_{\beta}(x)$.
Such hypothetical asymptotic analysis in the limit $\beta\to\infty$
will not produce an exact expression for the $\Ai_2$ process, but might
shed some light on the critical exponents $2/3$ and $1/3$, and might
help to explain the localization phenomenon for the path measure in the
strong disorder regime.

\begin{appendix}\label{app}

%s8 #&#
\section{Bounds on discrete random walk probabilities}

%le8.1 #&#
\begin{lemma}\label{lemma:discrete_bound}
There exists a constant $C > 0$ such that for all $k \geq0$ and $x
\in\R$,
\[
\sup_n \bigl\Vert n^{k/2} p_k^n
\bigr\Vert _{L^2} \leq C^k \Vert \varrho_k\Vert
_{L^2}, \qquad\sup_n\bigl \Vert n^{k/2}
p_{k|x}^n\bigr\Vert _{L^2} \leq C^k \Vert
\varrho_{k|x}\Vert _{L^2}.
\]
\end{lemma}

\begin{pf}
First observe that there exists a constant $C$ such that $\sqrt{i}
p(i,x) \leq C$ for all $i$ and $x$, and therefore
\[
\sup_{\vx\in\Z^k} p_k(\vi, \vx) \leq C^k
\prod_{j=1}^k \frac
{1}{\sqrt {\vi_j - \vi_{j-1}}}.
\]
From this and by definition of $p_k^n$ we have
\begin{eqnarray*}
\bigl\Vert n^{k/2} p_k^n\bigr \Vert _{L^2}^2
&=& n^k \sum_{\vi\in D_k^n} \sum
_{\vx\in
\Z
^k} p_k(\vi, \vx)^2
n^{-{3k}/{2}}
\\
&\leq& n^{-k/2} \sum_{\vi\in D_k^n} \max
_{\mathbf{y} \in\Z^k} p_k(\vi,\mathbf{y}) \sum
_{\vx\in\Z^k} p_k(\vi, \vx)
\\
&\leq &C^k n^{-k/2} \sum_{\vi\in D_k^n}
\prod_{j=1}^k \frac{1}{\sqrt {\vi
_j - \vi_{j-1}}}
\\
&=& C^k n^{-k} \sum_{\vi\in D_k^n}
\prod_{j=1}^k \biggl( \frac{\vi_j}{n} -
\frac{\vi_{j-1}}{n} \biggr)^{-1/2}
\\
&\leq& C^k \int_{\Delta_k} \prod
_{j=1}^k \frac{1}{\sqrt{\vt_j -
\vt
_{j-1}}} \,d \vt
\\
&\leq& \bigl(C' \bigr)^k \Vert \varrho_k
\Vert _{L^2}^2.
\end{eqnarray*}
The second-to-last inequality between the sum and the integral is an
easy consequence of $x \mapsto1/\sqrt{x}$ being a decreasing function,
and the sum being a ``right-endpoint'' approximation of the integral.
The second inequality is a simple extension of the first, using that
\[
p_{k|x}^n(\vt, \vx) = \frac{p_k^n(\vt, \vx) p_1^n(1 - \vt_k, x -
\vx
_k)}{p_1^n(1, x)}.
\]
\upqed\end{pf}

%s9 #&#
\section{Tightness}

In this section we consider the discrete time and space process
%
%e33 #&#
\begin{equation}
\label{eqn:tightness_discrete_time_space} \symz^{\omega}(k,x;\beta) = \Pl \Biggl[ \prod
_{i=1}^k \bigl(1 + \beta\omega (i,
S_i) \bigr) \mathbf{1} \{ S_k = x \} \Biggr].
\end{equation}
We will show that the continuous time and space processes
%
%e34 #&#
\begin{equation}
\label{eqn:tightness_cont_field} (t,x) \mapsto z_n^{\omega}(t,x):= \sqrt{n}
\symz^{\omega} \bigl(nt, x \sqrt{n}; n^{-1/4} \bigr)
\end{equation}
are tight as random elements in $C([\varepsilon, T] \times\R)$ for any $0
< \varepsilon< T < \infty$. Observe that since $\symz^{\omega}$ is
only defined
at lattice points, it requires some interpolation to extend
$z_n^{\omega
}$ to a continuous function of space and time. The exact interpolation
scheme will not really matter, but to be concrete we define it as
follows: at points $(t,x) \in[0,1] \times\R$ such that $(t,x)$ is a
corner point of the left-hand side of a rectangle in $\mathcal{R}_n$,
define $z_n^{\omega}$ according to \eqref{eqn:tightness_cont_field}.
Then for space--time points on the left edges of rectangles in $\mathcal
{R}_n$ define $z_n^{\omega}$ by linear interpolation of the values on
the corners that the edge connects, and finally for $(t,x)$ on the
interior points of rectangles define $z_n^{\omega}$ by linear
interpolation of the values at the four boundary corners.

The main idea for proving tightness will be to use equation \eqref
{eqn:tightness_discrete_time_space} to get a stochastic difference
equation for $\symz^{\omega}$, and then transfer this to a stochastic
difference equation for $z_n^{\omega}$ which approximates the
stochastic heat equation. Standard SPDE estimates from \cite
{walsh:st_flour}, which show the regularity in space and time of the
stochastic heat equation, are then shown to hold uniformly in $n$,
which will prove the tightness.

%re17 #&#
\begin{remark*}
The same argument will also show the tightness for the analogously
rescaled partition function
\[
e^{-n\lambda(\beta n^{-{1}/4})}\sqrt{n} Z^{\omega} \bigl(nt, x \sqrt{n}; \beta
n^{-1/4} \bigr)
\]
from the exponential model
\[
Z^\omega(k,x;\beta) = \Pl \Biggl[ \exp \Biggl\{ \beta\sum
_{i=1}^k \omega(i, S_i) \Biggr\}
\mathbf{1} \{ S_k = x \} \Biggr],
\]
because it can be written in the form (\ref
{eqn:tightness_discrete_time_space}) using the field $\tilde{\omega}_n$
as described in Proposition~\ref{prop:Z_p2l_convergence}.
\end{remark*}

%re18 #&#
\begin{remark*}
Tightness of the two-parameter field \eqref{eqn:tightness_cont_field}
is sufficient to prove tightness of the four-parameter field
%
%e35 #&#
\begin{equation}
\label{eqn:tightness_cont_4_field} (s,y;t,x) \mapsto\sqrt{n} \symz^{\omega} \bigl(ns, y
\sqrt{n}; nt, x \sqrt{n}; n^{-1/4} \bigr),
\end{equation}
where $\symz^{\omega}$ is defined by
\[
\symz^{\omega}(m,y; k, x; \beta) = \Pl \Biggl[  \prod
_{i=m+1}^{k} \bigl(1 + \beta \omega (i,
S_i) \bigr) \mathbf{1} \{ S_k = x \} \Big| S_m
= y \Biggr].
\]
Indeed, tightness of \eqref{eqn:tightness_cont_field} implies tightness
of \eqref{eqn:tightness_cont_4_field} in the forward $(t,x)$ variables,
and tightness in the $(s,y)$ variables follows from reversibility of
the random walk and the fact that the law of the environment field is
invariant under a similar time reversal. More precisely, define a field
$\omega_n$ by $\omega_n(i,x) = \omega(n-i, x)$. Then it is clear that
\begin{eqnarray*}
&&\bigl(1 + \beta\omega(n-m,y) \bigr) \symz^{\omega}(m,y;k,x;\beta)\\
&&\qquad =
\bigl(1 + \beta\omega (n-k,x) \bigr) \symz^{\omega_n}(n-k, x; n-m, y; \beta),
\end{eqnarray*}
hence the backward variables become the forward variables after
reversing the direction of time.
\end{remark*}

To obtain the difference equation for $\symz^{\omega}$ simply
condition on the
step between times $k$ and $k+1$ to obtain
%
%e36 #&#
\begin{eqnarray}
\label{eqn:tightness_bar} &&\symz^{\omega}(k+1, x; \beta)
\nonumber
\\[-8pt]
\\[-8pt]
\nonumber
&&\qquad = \tfrac{1}2 \bigl( 1
+ \beta\omega(k+1,x) \bigr) \bigl[ \symz^{\omega}(k, x+1; \beta) +
\symz^{\omega}(k, x-1;\beta) \bigr].
\end{eqnarray}
Throughout we let
\[
\overline{\symz^{\omega}}(k,x;\beta) = \tfrac{1}2 \bigl[
\symz^{\omega
}(k, x+1; \beta) + \symz^{\omega}(k, x-1; \beta) \bigr]
\]
so that equation \eqref{eqn:tightness_bar} becomes
\[
\symz^{\omega}(k+1, x; \beta) = \bigl(1 + \beta\omega(k+1, x) \bigr)
\overline{\symz^{\omega}}(k, x; \beta).
\]
Observe that the first and second terms of the right-hand side are
independent of each other. Subtracting $\symz^{\omega}(k,x;\beta)$
from both
sides of \eqref{eqn:tightness_bar} therefore yields
%
%e37 #&#
\begin{eqnarray}
\label{eqn:tightness_SHE} &&\symz^{\omega}(k+1, x; \beta) -\symz^{\omega}(k,x;
\beta)
\nonumber
\\[-8pt]
\\[-8pt]
\nonumber
&&\qquad= \tfrac{1}2 \Delta\symz^{\omega} (k,x;\beta) + \beta
\omega(k+1, x) \overline{\symz^{\omega
}}(k,x;\beta).
\end{eqnarray}
This is a discrete version of the stochastic heat equation, with
initial condition $\symz^{\omega}(0,x;\beta) = \mathbf{1} \{ x=0 \}
$. An immediate
advantage of equation \eqref{eqn:tightness_SHE} is that it allows for
an ``integral form'' representation of $\symz^{\omega}$. In general
if $Z$ is a
solution to
\[
Z(k+1,x) - Z(k,x) = \tfrac{1}2 \Delta Z(k,x) + f(k+1,x),
\]
then it is an easy consequence of the superposition principle that
\[
Z(k,x) = \E_{x} \bigl[ Z(0, S_k) \bigr] + \sum
_{i=1}^k \E_{x} \bigl[ f(i,
S_{k-i}) \bigr],
\]
where the expectation is over simple random walks $S$ beginning from
$x$. Applying this to equation \eqref{eqn:tightness_SHE} yields
\begin{eqnarray*}
\symz^{\omega}(k,x;\beta) &=& \Pl_x(S_k = 0) +
\beta\sum_{i=1}^k \E_x
\bigl[ \omega (i, S_{k-i}) \overline{\symz^{\omega}}(i-1,
S_{k-i}; \beta) \bigr]
\\
&= & p(k,x) + \beta\sum_{i=1}^k \sum
_y \omega(i,y) \overline{\symz
^{\omega}}(i-1, y; \beta) p(k-i, y-x).
\end{eqnarray*}
Now we translate this into an equation for the rescaled field (\ref
{eqn:tightness_cont_field}),
%
%e38 #&#
\begin{eqnarray}
\label{eqn:tightness_discrete_duhamel} z_n(t,x) &=& p_n(t,x)
\nonumber
\\[-8pt]
\\[-8pt]
\nonumber
&&{}+ n^{-3/2}
\mathop{\sum_{s\in[0,t]\cap
n^{-1}\mathbb{Z}}}_{ y \in n^{-1/2}\mathbb{Z}}
p_n(t-s,x-y)\overline{z}_n(s,y){\omega}(s+1/n,y),
\end{eqnarray}
where $\overline{z}_n$ is the rescaled analogue of $\overline{\symz
^{\omega}}$, and
$p_n(t,x) = \sqrt{n} \overline{p}_1(nt, x \sqrt{n}) = \sqrt{n}
p(\lceil nt \rceil, [x \sqrt{n}]_{\lceil nt \rceil})$ is the
rescaled point-to-point
transition probabilities for the random walk.
Observe that these $p_n$ are order one in $n$, and in fact converge to
$\varrho$ as $n \to\infty$. The key point, however, is the $n^{-3/2}$
in the second term, which is a consequence of the diffusive scaling in
space and time and the intermediate scaling on the environment. By
extending $\omega$ to a function that is piecewise constant on
rectangles of $\mathcal{R}_1$, we can rewrite \eqref
{eqn:tightness_discrete_duhamel} in an integral form
%
%e39 #&#
\begin{equation}
\label{78}\qquad z_n(t,x) = p_n(t,x) + \int
_0^t \int_{\mathbb{R}}
p_n(t-s,x-y) \overline {z}_n(s,y){\omega}(s + 1/n,y) \,dy
\,ds.
\end{equation}
A key fact is that the $\omega$ terms are independent of the
$\overline
{z}_n$ terms, since the term $\overline{z}_n(s, \cdot)$ is a function
of only the $\omega(i, \cdot)$ variables with $i \leq\lfloor ns
\rfloor
$. Using this fact one can derive the following a priori estimate, a
proof of which can be modified from Lemma~3.1 of \cite{ACQ}.

%le9.1 #&#
\begin{lemma} \label{apriori}Suppose that $z_n$ satisfies (\ref{78})
where the $\omega(i,x)$ are independent and
identically distributed with mean zero and $M$ finite moments, where $M
\geq2$. Then there exist $C, C_M$ such that for any $s>0$, $y \in\R$
and $n \geq1$,
\[
E \bigl[ z_n(s,y)^2 \bigr] \le C \varrho(s,y)^2,\qquad
E \bigl[ z_n(s,y)^M \bigr]\le C_M
\varrho(s,y)^M.
\]
The same bounds also extend to $\overline{z}_n$.
\end{lemma}

We will start the Duhamel formula (\ref{78}) with the data at time
$0<\varepsilon<t$,
%
%e40 #&#
\begin{eqnarray}
\label{78'} z_n(t,x) &=& \int p_n(t-
\varepsilon,x-y)z_n(\varepsilon, y) \,dy \nonumber\\
&&{}+ \int_\varepsilon
^t \int_{\mathbb{R}} p_n(t-s,x-y)
\overline{z}_n(s,y){\omega}(s + 1/n,y) \,dy \,ds
\\
&=:& A_{n,\varepsilon}(t,x) + U_{n,\varepsilon}(t,x).\nonumber
\end{eqnarray}
We develop modulus of continuity estimates for $A_{n,\varepsilon}(t,x)$
and $U_{n,\varepsilon}(x,t)$
based on Lemma~\ref{apriori}. We give the details for $U_{n,\varepsilon
}(x,t)$, as the analogous results for $A_{n,\varepsilon}(t,x)$
are much more straightforward. First, we use a discrete Burkholder
inequality (or equivalently the Marcinkiewicz--Zygmund inequality; see,
e.g., \cite{petrov}, page 61) which tells us that there is a constant
$C_M<\infty$ such that
\begin{eqnarray*}
&& E  \bigl[ \bigl| U_{n,\varepsilon}(x+\delta, t) - U_{n,\varepsilon}(x,t)\bigr|^M
\bigr] \\
&&\qquad\le
 C_M E \biggl[ \biggl( \int_\varepsilon^t
\int \bigl(p_n(t-s,x+\delta-y)\\
&&\hspace*{103pt}{}-p_n(t-s,x-y)
\bigr)^2 z_n(s,y)^2 \,dy \,ds
\biggr)^{M/2} \biggr].
\end{eqnarray*}
Apply H\"older's inequality with $p=M/2, q=M/(M-2)$ to bound this by
\begin{eqnarray*}
&&C'_M \biggl( \int_\varepsilon^t
\int \bigl(p_n(t-s,x+\delta-y)-p_n(t-s,x-y)
\bigr)^{{2M}/{(M-2)}} \,dy \,ds \biggr)^{{(M-2)}/{2}}\\
&&\qquad{}\times E \biggl[ \int
_\varepsilon^t \int z_n(s,y)^M
\,dy \,ds \biggr].
\end{eqnarray*}
From Lemma~\ref{apriori} we have a bound on the second term that is
independent of $n$
%
%e41 #&#
\begin{equation}
\label{eqn:tightness_expectation_bound} E \biggl[ \int_\varepsilon^t \int
z_n(s,y)^M \,dy \,ds \biggr] \le C_M.
\end{equation}
Note that this bound does depend on $t$ and $\varepsilon$, but we are
assuming that they are both within a compact interval that is bounded
away from zero and hence there is no issue. Further, one can check that
there is a $C$ also depending only on $t,\varepsilon>0$ such that
\begin{eqnarray*}
&&\biggl( \int_\varepsilon^t \int \bigl(p_n(t-s,x+
\delta-y)-p_n(t-s,x-y) \bigr)^{{2M}/{(M-2)}} \,dy \,ds
\biggr)^{{(M-2)}/{2}}\\
&&\qquad \le C_M \delta^{{M}/2 -1}.
\end{eqnarray*}
This is proved for the heat kernel in Walsh, and can be extended to the
present discrete case
by the local central limit theorem. Combining these estimates together
gives the existence of a constant $C_M$ such that
%
%e42 #&#
\begin{equation}
\label{eqn:x_bound} E \bigl[ \bigl| U_{n,\varepsilon}(x+\delta, t)
- U_{n,\varepsilon}(x,t)\bigr|^M
\bigr] \leq C_M \delta^{{M}/2 - 1}.
\end{equation}

We now produce similar estimates for $E[ | U_{n,\varepsilon}(x, t+h) -
U_{n,\varepsilon}(x,t)|^M]$. Writing out the difference we see that it
splits into the sum of two terms, and using the inequality $(a+b)^M
\leq2^M(a^M + b^M)$ together with the Burkholder inequality of above,
we upper bound the expectation by the sum of
%
%e43 #&#
\begin{eqnarray}
\label{first}  && C_M 2^M E \biggl[ \biggl( \int
_\varepsilon^t \int \bigl(p_n(t+h-s,x-y)
\nonumber
\\[-8pt]
\\[-8pt]
\nonumber
&&\hspace*{64pt}\qquad{}-p_n(t-s,x-y)
\bigr)^2 z_n(s,y)^2 \,dy \,ds
\biggr)^{{M}/2} \biggr]
\end{eqnarray}
and
%
%e44 #&#
\begin{equation}
\label{second} C_M 2^M E \biggl[ \biggl( \int
_t^{t+h} \int p_n(t+h-s,x-y)^2
z_n(s,y)^2 \,dy \,ds \biggr)^{{M}/2} \biggr].\vadjust{\goodbreak}
\end{equation}
The first term \eqref{first} we bound by H\"older's inequality with
$p=M/2, q=M/(M-2)$ to get
\begin{eqnarray*}
\hspace*{-4pt}&& E \biggl[ \int_\varepsilon^t \int
z_n(s,y)^M \,dy \,ds \biggr]\\
\hspace*{-4pt}&&\qquad{}\times  \biggl( \int
_\varepsilon^t \int\bigl| p_n(t+h-s,x-y)-p_n(t-s,x-y)\bigr|^{ {2M}/{(M-2)}}
\,dy \,ds \biggr)^{{(M-2)}/{2}}.
\end{eqnarray*}
The expectation term is bounded above by \eqref
{eqn:tightness_expectation_bound}. For $M > 6$, the second term is
uniformly bounded in $n$, again by using the local limit theorem to
perform the estimate for the heat kernel rather than the random walk
kernel. The upper bound is $C_M h^{{M}/{4} - 2}$.

We also use H\"{o}lder on \eqref{second} to get an upper bound of
\begin{eqnarray*}
&& E \biggl[ \int_t^{t+h} \int
z_n(s,y)^M \,dy \,ds \biggr]\\
&&\qquad{}\times  \biggl( \int
_{t}^{t+h} \int p_n(t+h-s,x-y)^{{2M}/{(M-2)}}
\,dy \,ds \biggr)^{{(M-2)}/{2}}.
\end{eqnarray*}
Again one uses the a priori bound of Lemma~\ref{apriori} and heat
kernel estimates to get that this is less than $C_M h^{{M}/4 - 2}$.
Therefore we have the existence of a constant $C_M$ such that
%
%e45 #&#
\begin{equation}
\label{eqn:t_bound} E \bigl[ \bigl| U_{n,\varepsilon}(x, t+h) -
 U_{n,\varepsilon}(x,t)\bigr|^M
\bigr] \leq C_M h^{{M}/4 - 2}.
\end{equation}
Combining \eqref{eqn:x_bound} with \eqref{eqn:t_bound} and the equation
\eqref{78'} we have the following:

%le9.2 #&#
\begin{lemma} \label{modcons} For each even $M>6$, and each $\varepsilon
>0$, there is a
$C_M<\infty$ such that for $t,t+h \ge\varepsilon$, and all $n \geq0$,
%
%e46 #&#
\begin{equation}
E \bigl[ \bigl| z_n(x+\delta, t+h) - z_n(x,t)\bigr|^M
\bigr]^{1/M}\le C_M \bigl(|\delta |+|h|\bigr)^{{1}/4 - {2}/{M}}.
\end{equation}
\end{lemma}

Now we use the inequality of Garsia \cite{garsia} that
%
%e47 #&#
\begin{equation}
\bigl|f(x)-f(y)\bigr| \le8 \int_0^{|x-y|}
\Psi^{-1} \bigl( B/ u^{2d} \bigr) \,dp(u)
\end{equation}
for all functions $f$ continuous in a unit cube $I\subset\mathbf{R}^d$
that satisfy the inequality
\[
\int_I\int_I\Psi \bigl(f(x)-f(y)
\bigr)/p \bigl(d^{-1/2} |x-y| \bigr) \,dx\,dy\leq B,
\]
where (i) $\Psi$ is nonconstant positive even convex function with
$\lim_{x\to\infty}\Psi(x)=\infty$, and (ii) $p$ is a positive
continuous even function increasing on $(0,\infty)$ that satisfies the
condition $\lim_{u\rightarrow0}p(u)=0$.

We are working in $d=2$ (space${}+{}$time). Choosing $\Psi(x)=x^M$, $M>6$
and $p(x)= x^{\gamma/M}$, we have from Lemma~\ref{modcons}.
%
%e48 #&#
\begin{equation}
E \biggl[ \int_{ t,s\in[\varepsilon,T], x,y\in\mathbb{R}} \Psi \biggl( \frac{ |z_n(x,
t) - z_n(y,s)|}{ p( 2^{-1/2} \sqrt{(t-s)^2+ (x-y)^2})} \biggr)
\biggr]\le C_M.
\end{equation}
Since
$
\int_0^{h} \Psi^{-1} ( B/ u^{2d}) \,dp(u) = C_{M,\gamma} B^{1/M}
h^{{(\gamma- 4)}/{M}}
$ with a finite $C_{M,\gamma}$ for $\gamma> 4$, we conclude that if
$\mathcal{H}_{[\varepsilon,T]\times\mathbb{R}} (\alpha, K)$
denotes the set of functions $z(t,x)$ on $[\varepsilon,T]\times\mathbb{R}$
with $|z(t,x)-z(s,y)|\le
K |(t-s)^2+ (x-y)^2|^{ \alpha/2}$, then we have the following:

%le9.3 #&#
\begin{lemma} If $P_n$ denotes the distribution of $z_n(t,x)$, then for
any $\varepsilon>0$ and
$\alpha<1/4$,
%
%e49 #&#
\begin{equation}
\limsup_{K\to\infty} \limsup_{n\to\infty}
P_n \bigl( C \bigl([\varepsilon, T]\times \mathbb{R} \bigr)\setminus
\mathcal{H}_{[\varepsilon,T]\times\mathbb{R}} (\alpha, K) \bigr)=0.
\end{equation}
In particular, since $\mathcal{H}_{[\varepsilon,T]\times\mathbb{R}}
(\alpha
, K)$ are compact sets of $C([\varepsilon, T]\times\mathbb{R}) $, the
$P_n$ are tight.
\end{lemma}

%re19 #&#
\begin{remark*} If we had been more careful we could improve the modulus
of continuity to H\"older $1/2-$ in space, but for tightness
we do not need an optimal result.
\end{remark*}
\end{appendix}

\section*{Acknowledgments} The authors would like to thank Ivan Corwin
and several anonymous referees for helpful comments and literature
references that led to a much improved presentation.

% imsref loaded by akundreckaite, 2013-12-05 13:58:03
% imsref loaded by akundreckaite, 2013-12-05 13:59:39
% imsref loaded by akundreckaite, 2013-12-05 14:31:04


%suskaldyti doi

\printaddresses

\end{document}